\documentclass{amsart}

\usepackage{amssymb}

\usepackage{epsfig}         

\newtheorem{thm}{Theorem}[section]
\newtheorem{prop}[thm]{Proposition}

\newtheorem{cor}[thm]{Corollary}

\theoremstyle{definition}
\newtheorem{definition}[thm]{Definition}

\theoremstyle{remark}

\newtheorem{remark}[thm]{Remark}

\numberwithin{equation}{section}
\newcommand{\N}{\mathbb{N}}  
\newcommand{\R}{\mathbb{R}}  
\newcommand{\C}{\mathbb{C}}  

\renewcommand{\Re}{\mbox{ \rm Re}}

\begin{document}

\title{\sc Fr\'echet differentiability in Fr\'echet spaces, and differential equations with unbounded variable delay}

\author{Hans-Otto Walther}

\address{Mathematisches Institut, Universit\"{a}t Gie{\ss}en,
Arndtstr. 2, D 35392 Gie{\ss}en, Germany. E-mail {\tt
Hans-Otto.Walther@math.uni-giessen.de}, phone ++49-151-58546608,
fax ++49-641-9932029}

\date{January 26, 2018}

\begin{abstract}

We introduce and discuss Fr\'echet differentiability for maps between Fr\'echet spaces. For delay differential equations $x'(t)=f(x_t)$ we construct a continuous semiflow of continuously differentiable solution operators $x_0\mapsto x_t$, $t\ge0$, on submanifolds of the Fr\'echet space $C^1((-\infty,0],\R^n)$, and establish local invariant manifolds at stationary points by means of transversality and embedding properties. The results apply to examples with unbounded
but locally bounded delay.

\bigskip

\noindent
MSC 2010: 34 K 05, 37 L 05

\medskip


\end{abstract}

\maketitle

\section{Introduction}

Consider an autonomous  delay differential equation
\begin{equation}
x'(t)=f(x_t)
\end{equation}
with $f:U\to\R^n$ defined on a set of maps $(-\infty,0]\to\R^n$, and the segment, or history, $x_t$ of the solution $x$ at $t$ defined by $x_t(s)=x(t+s)$ for all $s\le0$. A solution on some interval $[t_0,t_e)$, $t_0<t_e\le\infty$, is a map $x:(-\infty,t_e)\to\R^n$ with $x_t\in U$ for all $t\in[t_0,t_e)$ so that the restriction of $x$ to $[t_0,t_e)$ is differentiable and Eq. (1.1) holds on this interval. Solutions on the whole real line are defined accordingly.
A toy example which can be written in the form (1.1) is the equation
\begin{equation}
x'(t)=h(x(t-r)),\quad r=d(x(t))
\end{equation}
with functions $h:\R \to\R$ and $d:\R\to[0,\infty)$. Other examples arise from pantograph equations
\begin{equation}
x'(t)=a\,x(\lambda t)+b\,x(t)
\end{equation}
with constants $a\in\C$, $b\in\R$ and $0<\lambda<1$, and from 
Volterra integro-differential equations
\begin{equation}
x'(t)=\int_0^tk(t,s)h(x(s))ds
\end{equation}
with $k:\R^{n\times n}\to\R$ and $h:\R^n\to\R^n$ continuous \cite{W9}. Eq. (1.3) is linear, and both equations (1.3) and (1.4) are non-autonomous. We shall come back to them in Section 9 below.

\medskip

Building a theory of Eq. (1.1) which (a) covers examples with state-dependent delay like Eq. (1.2) and (b) results in solution operators $x_0\mapsto x_t$, $t\ge0$,  which are continuously differentiable begins with the search for a suitable state space. For equations with bounded delay the basic steps of a solution theory were made in \cite{W1}, starting from the observation that the domain of the functional on the right hand side of the differential equation must consist of maps which are continuously differentiable, and not merely continuous as in the by-now well established theory of retarded functional differential equations \cite{HVL,DvGVLW}. Accordingly the functional $f$ in Eq. (1.1) above should be defined on a subset $U$ of the vector space $C^1=C^1((-\infty,0],\R^n)$ of continuously differentiable maps $(-\infty,0]\to\R^n$. 
Linearization as in \cite{W1} suggests that in the new theory autonomous linear equations with constant delay, like for example,
$$
x'(t)=-\alpha\,x(t-1)
$$
will appear,
which have solutions on $\R$ with arbitrarily large exponential growth at $-\infty$. In order not to loose such
solutions we stay with the full space $C^1$ and work with the topology of locally uniform convergence of maps and their derivatives, which makes $C^1$ a Fr\'echet space.

\medskip

In \cite{W7} we saw that under mild smoothness hypotheses on $f$, which hold in examples with state-dependent delay, the set
$$
X_f=\{\phi\in U:\phi'(0)=f(\phi)\}
$$ 
is a continuously differentiable submanifold of codimension $n$ in $C^1$. Notice that $X_f$ consists of the segments $x_t$, $0\le t<t_e$, of all continuously differentiable solutions $x:(-\infty,t_e)\to\R^n$ on $[0,t_e)$, $0<t_e\le\infty$, of Eq. (1.1). It is shown in \cite{W7} that these solutions constitute a continuous semiflow $(t,x_0)\mapsto x_t$ on $X_f$, with continuously differentiable solution operators $x_0\mapsto x_t$, $t\ge0$. Here continuous differentiability is understood in the sense of Michal \cite{M} and Bastiani \cite{B}, which means for a continuous map
$f:V\supset U\to W$, $V$ and $W$ topological vector spaces and $U\subset V$ open, that all directional derivatives 
$$
Df(u)v=\lim_{0\neq t\to0}\frac{1}{t}(f(u+tv)-f(u))
$$
exist and that the map
$$
U\times V\ni(u,v)\mapsto Df(u)v\in W
$$
is continuous. Let us briefly speak of $C^1_{MB}$-smoothness. 

\medskip

It is convenient to call the set $X_f$ the solution manifold associated with the map $f$.

\medskip

The mild hypotheses on $f$ mentioned above are that $f$ is $C^1_{MB}$-smooth and that 

\medskip

(e) {\it each derivative $Df(\phi):C^1\to\R^n$, $\phi\in U$, has a linear extension $D_ef(\phi):C\to\R^n$, with
the map
$$
U\times C\ni(\phi,\chi)\mapsto D_ef(\phi)\chi\in\R^n
$$
being continuous.}

\medskip

Here $C$ is the Fr\'echet space of continuous maps $(-\infty,0]\to\R^n$ with the topology of locally uniform convergence.
Property (e) is closely related to the earlier notion of being {\it almost Fr\'echet differentiable} from \cite{M-PNP}, for maps on a Banach space of continuous functions.

\medskip

An inspection of examples of differential equations with state-dependent delay for which the map $f$ in Eq. (1.1) is $C^1_{MB}$-smooth reveals that in these examples $f$ is in fact better, namely, that it is continuously differentiable in the sense of
the following definition.

\begin{definition}
A continuous map $f:V\supset U\to W$, $V$ and $W$ topological vector spaces and $U\subset V$ open, is said to be $C^1_F$-smooth if all directional derivatives exist, if each map $Df(u):V\to W$, $u\in U$, is linear and continuous, and if the map $Df:U\ni u\mapsto Df(u)\in L_c(V,W)$ is continuous with respect to the topology $\beta$ of uniform convergence on bounded sets, on the vector space $L_c(V,W)$ of continuous linear maps $V\to W$.
\end{definition}

The letter $F$ in the symbol $C^1_F$ stands for Fr\'echet because in case $V$ and $W$ are Banach spaces
$C^1_F$-smoothness is equivalent to the familiar notion of continuous differentiability with Fr\'echet derivatives,
see e. g. Proposition 3.4  below. In case $V$ and $W$ are Fr\'echet spaces $C^1_F$-smoothness is equivalent to $C^1_{MB}$-smoothness combined with the continuity of the derivative with respect to the topology $\beta$ on $L_c(V,W)$, see e. g. Proposition 3.2 below.

\medskip

It seems that $C^1_F$-smoothness of maps in Fr\'echet spaces which are not Banach spaces  has not attracted much attention, compared to $C^1_{MB}$-smoothness and further notions of smoothness \cite{SL}. For possible reasons, see \cite[Chapter II, Section 3]{B}. In any case, for the study of Eq. (1.1) the notion of $C^1_F$-smoothness is useful - and yields, of course,  slightly stronger results, compared to the theory based on $C^1_{MB}$-smoothness in \cite{W7,W8}.  The present report shows how to obtain solution manifolds, solution operators, and local invariant manifolds at stationary points, all of them $C^1_F$-smooth, and discusses Eqs. (1.2)-(1.4) as examples. We mention in this context that we do not touch upon higher order derivatives -  in \cite{KW} it is shown that solution manifolds are in general not better than $C^1_F$-smooth. The same holds for infinite-dimensional local invariant manifolds in the solution manifold, whereas finite-dimensional local invariant manifolds may be $k$ times continuously differentiable, $k\in\N$, under appropriate hypotheses on the map $f$ in Eq. (1.1) \cite{K1}.

\medskip

The present paper is divided into 3 parts. Part I with Sections 2-8 is about $C^1_F$-maps in general. Sections 2-4 collect what we need to know about $C^1_F$-maps, beginning in Section 2 with the topology $\beta$ on $L_c(V,W)$ for topological spaces $V,W$. Section 3 discusses $C^1_F$-smoothness for maps between Fr\'echet spaces, and its relations to $C^1_{MB}$-smoothness, and provides elements of calculus, including the chain rule for $C^1_F$-maps. In order to keep Section 3 short we make extensive use of \cite[Part I]{H} on $C^1_{MB}$-smoothness.  Section 4 is about $C^1_F$-submanifolds of Fr\'echet spaces. Sections 5-7 contain a uniform contraction principle, an implicit function theorem, and simple transversality- and embedding results which yield $C^1_F$-submanifolds of finite dimension or finite codimension. All of these results are familiar in the Banach space case, and most of them are well-known also in the $C^1_{MB}$-setting \cite{G1,G2,W7,W8}. In Section 8 examples illustrate the difference between $C^1_{MB}$-smoothness and $C^1_F$-smoothness. None of them is related to a delay differential equation. 

\medskip

Part II with Sections 9-12 is about the construction of the semiflow on the solution manifold of Eq.(1.1) for $C^1_F$-maps $f:C^1\supset U\to\R^n$ which have property (e).  In Section 9 these hypotheses are verified for maps $f$ related to the examples (1.2)-(1.4). Notice that for the Volterra equation (1.4) the associated  $C^1_F$-map $f$ is defined on the big space $C\supset C^1$.  - Proposition 9.5 guarantees 
that the set $X_f$ is indeed a $C^1_F$-submanifold of codimension $n$ in the space $C^1$. The proof is by Proposition 7.1 on transversality. Sections 10-12 establish a continuous semiflow on the solution manifold whose solution operators are $C^1_F$-smooth. This is parallel to the approach in \cite{W7}, and we only describe how to modify parts of \cite{W7}  in order to obtain the present result, which is stated in Section 12.

\medskip

Part III with Sections 13-17  is based on \cite{W8}. We explain how to modify constructions in \cite{W8},  in order to obtain local invariant manifolds at stationary points of the semiflow which are $C^1_F$-smooth, by means of the transversality- and embedding results from Section 7. An important ingredient is
\cite[Proposition 1.2]{W7} which says that a $C^1_{MB}$-map $f:C^1\supset U\to\R^n$ is of locally bounded delay, in the sense that

\medskip

(lbd) {\it for every $\phi\in U$ there are a neighbourhood $N(\phi)\subset U$ and $d>0$ such that for all $\chi,\psi$ in $N(\phi)$ with
$$
\chi(t)=\psi(t)\quad\text{for all}\quad t\in[-d,0]
$$
we have $f(\chi)=f(\psi)$.}

\medskip

It follows from property (lbd) that solutions of Eq. (1.1) with segments close to a stationary point $\bar{\phi}\in X_f\subset U$ are given by solutions of an equation
\begin{equation}
x'(t)=f_d(x_t)
\end{equation}
with a map $f_d:U_d\to\R^n$ defined on an open neighbourhood $U_d$ of the restriction $\bar{\phi}|_{[-d,0]}$, in the Banach space $C^1_d$ of continuously differentiable maps $[-d,0]\to\R^n$, and with the segments $x_t$ defined on $[-d,0]$. 

\medskip

The solutions of Eq. (1.5) generate a semiflow on the solution manifold $X_{f_d}=\{\psi\in U_d:\psi'(0)=f_d(\psi)\}$ in the Banach space $C^1_d$. This will be exploited in the search for local invariant manifolds in the Fr\'echet space $C^1$. For example, in Section 15 a local stable manifold at a stationary point $\bar{\phi}$ of the semiflow on $X_f\subset C^1$ appears as a preimage of the restriction map $C^1\ni\phi\mapsto\phi|_{[-d,0]}\in C^1_d$
which is transversal to a local stable manifold at $\bar{\phi}|_{[-d,0]}$ in $X_{f_d}\subset C^1_d$. The latter was obtained in \cite[Section 3.5]{HKWW}.

\medskip

In  the earlier result in \cite{W8}, on local invariant manifolds which are $C^1_{MB}$-smooth, it was necessary to add a technical hypothesis (d) which essentially requires that $f_d$ is $C^1_F$-smooth. Notice that in the present approach where $f$ is $C^1_F$-smooth and not only $C^1_{MB}$-smooth the hypothesis (d) is obsolete.

\medskip

For other work on delay differential equations with states $x_t$ in Fr\'echet spaces see \cite{Sch,S,W9}.

\medskip

{\bf Notation, preliminaries.} The closure of a subset $M$ of a topological space is denoted by $\overline{M}$ and its interior is denoted by $\stackrel{\circ}{M}$. 

\medskip

For basic facts about  topological vector spaces see \cite{R}. The vector space of continuous linear maps $V\to W$ between topological vector spaces is denoted by $L_c(V,W)$.

\medskip

Recall that a subset $B\subset T$ of a topological vector space over the field $\mathbb{K}$, $\mathbb{K}=\R$ or $\mathbb{K}=\C$, is bounded if for every neighbourhood $N$ of $0\in T$ there exists a real $r_N\ge0$ with $B\subset rN$ for all reals $r\ge r_N$.  The points of convergent sequences form bounded sets, compact sets are bounded. A set $A\subset T$ is balanced if  $zA\subset A$ for all $z\in\mathbb{K}$ with $|z|\le1$. If $A$ is balanced and $|z|\ge1$ then $A\subset zA$.

\medskip

Continuous linear maps between topological vector spaces map bounded sets into bounded sets.

\medskip

Products of topological vector spaces are always equipped with the product topology.

\begin{prop}  
\cite[Proposition 1.2]{W9} Suppose $T$ is a topological space, $W$ is a topological vector space, $M$ is a metric space with metric $d$, $g:T\times M\supset U\to W$ is continuous, $U\supset\{t\}\times K$, $K\subset M$ compact. Then $g$ is uniformly continuous on $\{t\}\times K$ in the following sense: For every neighbourhood $N$ of $0$ in $W$ there exist a neighbourhood $T_N$ of $t$ in $T$ and $\epsilon>0$ such that for all $t'\in T_N$, all $\hat{t}\in T_N$,  all $k\in K$, and all $m\in M$ with
$$
d(m,k)<\epsilon\quad\text{and}\quad (t',k)\in U,\quad (\hat{t},m)\in U
$$
we have
$$
g(t',k)-g(\hat{t},m)\in N.
$$
\end{prop}

\begin{proof}
Choose a neighbourhood $N'$ of $0$ in $W$ with $N'+N'\subset N$.  For every $k\in K$ there exist open neighbourhoods
$T(k)$ of $t$ in $T$ and $\delta(k)>0$ such that for all $t'\in T(k)$ and all $m\in M$ with $d(m,k)<\delta(k)$ and
$(t',m)\in U$ we have 
$$
g(t',m)-g(t,k)\in N',
$$
due to continuity. The compact set $K$ is contained in a finite union of open neighbourhoods
$$
\left\{m\in M:d(m,k_j)<\frac{\delta(k_j)}{2}\right\},\quad j=1,\ldots,n,
$$
with $k_1,\ldots,k_n$ in $K$. Set 
$$
\epsilon=\min\left\{\frac{\delta(k_j)}{2}:j=1,\ldots,n\right\}\quad\text{and}\quad T_N=\cap_{j=1}^nT(k_j).
$$ 
Let $t'\in T_N$, $\hat{t}\in T_N$,  $k\in K$, and $m\in M$ with
$$
d(m,k)<\epsilon,\quad (t',k)\in U,\quad (\hat{t},m)\in U
$$
be given. For some $j\in\{1,\ldots,n\}$, $d(k,k_j)<\frac{\delta(k_j)}{2}$. By the triangle inequality, $d(m,k_j)<\delta(k_j)$. It follows that
$$
g(t',k)-g(\hat{t},m)=(g(t',k)-g(t,k_j))+(g(t,k_j)-g(\hat{t},m))\in N'+N'\subset N.
$$
\end{proof}

A Fr\'echet space $F$ is a locally convex topological vector space which is complete and metrizable. The topology is given by a sequence of seminorms 
$|\cdot|_j$, $j\in\N$, which are separating in the sense that $|v|_j=0$ for all $j\in\N$ implies $v=0$. The sets
$$
N_{j,k}=\left\{v\in F:|v|_j<\frac{1}{k}\right\},\quad j\in\N\quad\text{and}\quad k\in\N,
$$
form a neighbourhood base at the origin. If the sequence of seminorms is increasing then the sets
$$
N_j=\left\{v\in F:|v|_j<\frac{1}{j}\right\},\quad j\in\N,
$$
form a neighbourhood base at the origin.

\medskip

Products of Fr\'echet spaces, closed subspaces of Fr\'echet spaces, and Banach spaces are Fr\'echet spaces.

\medskip

For a curve, a continuous map $c$ from an interval $I\subset\R$ of positive length into a Fr\'echet space $F$, the
tangent vector at $t\in I$ is 
$$
c'(t)=\lim_{0\neq h\to0}\frac{1}{h}(c(t+h)-c(t))
$$
provided the limit exists. As in \cite[Part I]{H} the curve is said to be continuously differentiable if it has tangent vectors everywhere and if the map
$$
c':I\ni t\mapsto c'(t)\in F
$$
is continuous.

\medskip

For a continuous map $f:V\supset U\to F$, $V$ and $F$ Fr\'echet spaces and $U\subset V$ open, and for $u\in U,v\in V$ the directional derivative is defined by
$$
Df(u)v=\lim_{0\neq h\to0}\frac{1}{h}(f(u+hv)-f(u))
$$
provided the limit exists. If for $u\in U$ all directional derivatives $Df(u)v$, $v\in V$ exists then the map $Df(u):V\ni v\mapsto Df(u)v\in F$ is called the derivative of $f$ at $u$. 

\medskip

For continuous maps $f:U\to F$, $V,W,F$ Fr\'echet spaces and $U\subset V\times W$ open,  partial derivatives are defined in the usual way. For example, $D_1f(v,w):V\to F$ is given by
$$
D_1f(v,w)\hat{v}=\lim_{0\neq h\to0}\frac{1}{h}(f(v+h\hat{v},w)-f(v,w)).
$$

The tangent cone of a set $M\subset F$, $F$ a Fr\'echet space, at $x\in M$ is the set $T_xM$ of all tangent vectors $v=c'(0)$ of continuously differentiable curves $c:I\to F$ with $I$ open, $0\in I$, $c(0)=x$,  $c(I)\subset M$.

\medskip

We freely use facts about the Riemann integral for continuous maps $[a,b]\to F$  into a Fr\'echet space and results from calculus based on continuous differentiability in the sense of Michal and Bastiani which can be found in  \cite[Sections I.1-I.4]{H}. 

\medskip

For maps $\R^k\supset U\to\R^n$, $U\subset\R^k$ open, $C^1_F$-smoothness and $C^1_{MB}$-smoothness are equivalent (see Propositions 3.2 and 3.3 below, for example), and we simply speak of continuously differentiable maps. Also for a curve $c:I\to\R^n$ on an interval  $I\subset\R$ of positive length and not open, we only speak of continuous differentiability, with $c'(t)=\lim_{0\neq h\to0}\frac{1}{h}(c(t+h)-c(t))=Dc(t)1\in\R^n$ at inner points and one-sided derivatives $c'(t)$  at endpoints contained in $I$.

\medskip

In Part II the following Fr\'echet spaces are used: For $n\in\N$ and $k\in\N_0$ and $T\in\R$, $C^k_T=C^k((-\infty,T],\R^n)$ denotes the Fr\'echet space of $k$-times continuously differentiable maps $\phi:(-\infty,T]\to\R^n$ with the seminorms given by
$$
|\phi|_{k,T,j}=\sum_{\kappa=0}^k\max_{T-j\le t\le T}|\phi^{(\kappa)}(t)|,\quad j\in\N,
$$ 
which define the topology of uniform convergence of maps and their derivatives on compact sets. Analogously we consider the space $C^k_{\infty}=C^k(\R,\R^n)$, with 
$$
|\phi|_{k,\infty,j}=\sum_{\kappa=0}^k\max_{-j\le t\le j}|\phi^{(\kappa)}(t)|.
$$ 
In case $T=0$ we abbreviate $C^k=C^k_0$, $|\cdot|_{k,j}=|\cdot|_{k,0,j}$, and also $C=C^0=C^0_0$, $|\cdot|_j=|\cdot|_{0,j}=|\cdot|_{0,0,j}$. In case $T=\infty$ we abbreviate $C_{\infty}=C^0_{\infty}$, $|\cdot|_{\infty,j}=|\cdot|_{0,\infty,j}$.

\medskip

The vector space $C^{\infty}=\cap_{k=0}^{\infty}C^k$ will be used without a topology on it.

\medskip

The differentiation map $\partial_{k,T}:C^k_T\ni\phi\mapsto\phi'\in C^{k-1}_T$, $k\in\N$ and $T\in\R$ or $T=\infty$,  is linear and continuous. We abbreviate $\partial_T=\partial_{1,T}$ and $\partial=\partial_0=\partial_{1,0}$.

\medskip

The following Banach spaces occur in Parts II and III:  For $n\in\N$ and $k\in\N_0$ and reals $a<b$ , $C^k([a,b],\R^n)$ denotes the Banach space of $k$-times continuously differentiable maps $[a,b]\to\R^n$ with the norm given by
$$
|\phi|_{[a,b],k}=\sum_{\kappa=0}^k\max_{a\le t\le b}|\phi^{(\kappa)}(t)|.
$$ 
We use various abbreviations, for $S<T$ and $d>0$ and $k\in\N_0$:
\begin{eqnarray*}
C_{ST} &= & C^0([S,T],\R^n),\quad |\cdot|_{ST}=|\cdot|_{[S,T],0}\\
C^1_{ST} &= & C^1([S,T],\R^n),\quad |\cdot|_{1,ST}=|\cdot|_{[S,T],1}\\
C^k_d & = & C^k([-d,0],\R^n),\quad |\cdot|_{d,k}=|\cdot|_{[-d,0],k}
\end{eqnarray*}
It is easy to see that the linear restriction maps
\begin{equation*}
R_{d,k}:C^k\to C^k_d,\quad d>0\quad\text{and}\quad k\in\N_0,
\end{equation*}
and the linear prolongation maps
\begin{equation*}
P_{d,k}:C^k_d\to C^k,\quad d>0\quad\text{and}\quad k\in\N_0,
\end{equation*}
given by $(P_{d,k}\phi)(s)=\phi(s)$ for $-d\le s\le0$ and
\begin{equation*}
(P_{d,k}\phi)(s)=\sum_{\kappa=0}^k\frac{\phi^{(\kappa)}(-d)}{\kappa!}(s+d)^{\kappa}\quad\text{for}\quad s<-d
\end{equation*}
are continuous, and for all $d>0$ and $k\in\N_0$,
\begin{equation*}
R_{d,k}\circ P_{d,k}=\text{id}_{C^k_d}.
\end{equation*}

In Part II we also need the closed subspaces
$$
C_{0T,0}=\{\phi\in C_{0T}:\phi(0)=0\}\quad\text{and}\quad C^1_{0T,0}=\{\phi\in C^1_{0T}:\phi(0)=0=\phi'(0)\}.
$$
In Part 3 we also need the Banach space $B_a$, for $a>0$ given, of all $\phi\in C$ with
$$
\sup_{t\le0}|\phi(t)|e^{at}<\infty,\quad|\phi|_a=\sup_{t\le0}|\phi(t)|e^{at},
$$
and finally, the Banach space $B^1_a$ of all $\phi\in C^1$ with
$$
\phi\in B_a,\quad\phi'\in B_a,\quad|\phi|_{a,1}=|\phi|_a+|\phi'|_a.
$$
Solutions of equations
\begin{equation*}
x'(t)=g(x_t),\quad\text{with}\quad g:C^1_d\supset U\to\R^n\quad\text{or}\quad g:B^1_a\supset U\to\R^n,
\end{equation*}
on some interval $I\subset\R$ are defined as in case of Eq. (1.1): With $J=[-d,0]$ or $J=(-\infty,0]$, respectively, they are continuously differentiable maps $x:J+I\to\R^n$ so that $x_t\in U$ for all $t\in I$ and the differential equation holds for all $t\in I$. Observe that $x_t$ may denote a map on $[-d,0]$ or on $(-\infty,0]$, depending on the context.

\medskip

The following statement on ``globally bounded delay'' for continuous linear maps corresponds to a special case of  \cite[Proposition 1.2]{W7}.

\begin{prop}\cite[Proposition 1.2]{W8}
For every continuous linear map $L:C\to B$, $B$ a Banach space, there exists $r>0$ with
$L\phi=0$ for all $\phi\in C$ with $\phi(s)=0$ on $[-r,0]$.
\end{prop}

For results on strongly continuous semigroups given by solutions of linear autonomous retarded functional differential equations
\begin{equation*}
x'(t)=\Lambda x_t
\end{equation*}
with $\Lambda:C_d\to\R^n$ linear and continuous, see \cite{DvGVLW,HVL}.

\newpage

\begin{center}
{Part I}
\end{center}

\medskip

\section{Uniform convergence of continuous linear maps on bounded sets}

\medskip

Let $V,W$ be topological vector spaces over $\R$ or $\C$. On $L_c=L_c(V,W)$ the topology $\beta$  of uniform convergence on bounded sets is defined as follows.
For a neighbourhood $N$ of $0$ in $W$ and a bounded set $B\subset V$ the neighbourhood $U_{N,B}$ of $0$ in $L_c$ is defined as
$$
U_{N,B}=\{A\in L_c:AB\subset N\}.
$$
Every finite intersection of such sets $U_{N_j,B_j}$, $j\in\{1,\ldots,J\}$, contains a set of the same kind, because we have
$$
\cap_{j=1}^JU_{N_j,B_j}\supset\{A\in L_c:A(\cup_{j=1}^JB_j)\subset\cap_{j=1}^JN_j\},
$$
finite unions of bounded sets are bounded, and finite intersections of neighbourhoods of $0$ are neighbourhoods of $0$. Then the topology $\beta$ is the set of all $O\subset L_c$ which have the property that for each $A\in O$ there exist a neighbourhood $N$ of $0$ in $W$ and a bounded set $B\subset V$ with
$A+U_{N,B}\subset O$. It is the easy to show that indeed $\beta$  is a topology.

\medskip

We call a map $A$ from a topological space $T$ into $L_c$ $\beta$-continuous at a point $t\in T$ if it is continuous at $t$ with respect to the topology $\beta$ on $L_c$.

\begin{remark} 
(i) Convergence of a sequence in $L_c$ with respect to $\beta$ is equivalent to uniform convergence on every bounded subset $B\subset V$. (Proof: By definition convergence $A_j\to A$ with respect to $\beta$ is equivalent to convergence $A_j-A\to0$ with respect to $\beta$. This means that for each neighbourhood $N$ of $0$ in $W$ and for each bounded subset $B\subset V$ there exists $j_{N,B}\in\N$ so that for all integers $j\ge j_{N,B}$, $A_j-A\in U_{N,B}$. Or, for all integers $j\ge j_{N,B}$ and all $b\in B$, $(A_j-A)b\in N$. Now the assertion becomes obvious.)

\medskip

(ii) If $V$ and $W$ are  Banach spaces then $\beta$ is the norm topology on $L_c(V,W)$ given by $|A|=\sup_{|v|\le1}|Av|$. 

\medskip

(iii) In order to verify $\beta$-continuity of a map $A:T\to L_c$, $T$ a topological space, at some $t\in T$ one has to show that, given a bounded subset $B\subset V$ and a neighbourhood $N$ of $0$ in $W$, there exists a neighbourhood $N_t$ of $t$ in $T$  such that for all $s\in N_t$ we have $(A(s)-A(t))(B)\subset N$. 

\medskip

In case $T$ has countable neighbourhood bases 
the map $A$ is $\beta$-continuous  at $t\in T$ if and only if for any sequence $T\ni t_j \to t$ we have $A(t_j)\to A(t)$. For $A(t_j)\to A(t)$ we need that given a bounded subset $B\subset V$ and a neighbourhood $N$ of $0$ in $W$, there exists $J\in\N$ with
$$
(A(t_j)-A(t))(B)\subset N\quad\text{for all integers}\quad j\ge J.
$$
In the sequel we shall use the previous statement frequently.
\end{remark}

\begin{prop} 
Singletons $\{A\}\subset L_c$ are closed with respect to the topology $\beta$, and $L_c$ equipped with the topology $\beta$ is a topological vector space.
\end{prop}

\begin{proof} 1. (On singletons) Let $A\in L_c$ be given. We show $L_c\setminus\{A\}\in\beta$. Let $S\in L_c\setminus\{A\}$. For some $b\in V$, $Ab\neq Sb$. For some neighbourhood $N$ of $0$ in $W$, $Ab\notin Sb+N$ \cite[Theorem 1.12]{R}. For all $S'\in U_{N,\{b\}}$ we have $S'b\in N$, hence
$(S+S')b\in Sb+N$, and thereby, $S+S'\neq A$. Hence $S+U_{N,\{b\}}\subset L_c\setminus\{A\}$.

\medskip

2. (Continuity of addition) Assume $S,T$ in $L_c$ and let $U_{N,B}$ be given, $N$ a neighbourhood of $0$ in $W$ and $B$ a bounded subset of $V$. We have to find neighbourhoods of $S$ and $T$ so that addition maps their Cartesian product
into $S+T+U_{N,B}$. As $W$ is a topological vector space there are neighbourhoods $N_T,N_S$
of $0$ in $W$ with $N_T+N_S\subset N$. For every $T'\in T+U_{N_T,B}$ and for every $S'\in S+U_{N_S,B}$ and for every $b\in B$ we get $((T'+S')-(T+S))b=T'b-Tb+S'b-Sb\in N_T+N_S\subset N$, hence $((T'+S')-(T+S))B\subset N$, or
$T'+S'\in T+S+U_{N,B}$.

\medskip

3. (Continuity of multiplication with scalars) In case of vector spaces over the field $\C$ let $c\in\C,T\in L_c$. Let a neighbourhood $N$ of $0$ in $W$ and a bounded subset $B\subset V$ be given, and consider the neighbourhood $U_{N,B}$ of $0$ in $L_c$. There is a neighbourhood $\hat{N}$ of $0$ in $W$ with $\hat{N}+\hat{N}+\hat{N}\subset N$, see e. g. \cite[proof of Theorem 1.10]{R}. We may assume that $\hat{N}$ is balanced \cite[Theorem 1.14]{R}. As $TB$ is bounded there exists $r_N\ge0$ such that for reals $r\ge r_N$, $TB\subset r\hat{N}$. 
We infer that for some $\epsilon>0$, $(0,\epsilon)TB\subset\hat{N}$. 
As $\hat{N}$ is balanced we obtain that for all $z\in\C$ with $0<|z|<\epsilon$, 
$$
zTB=z\frac{|z|}{|z|}TB=\frac{z}{|z|}|z|TB\subset\frac{z}{|z|}\hat{N}\subset\hat{N}.
$$
With $0\in\hat{N}$ we arrive at
$$
zTB\subset\hat{N}\quad\text{for all}\quad z\in\C\quad\text{with}\quad|z|<\epsilon.
$$
Because of continuity of multiplication $\C\times W\to W$ there are neighbourhoods $U_c$ of $0$ in $\C$ and $N_c$ of $0$ in $W$ with
$$
U_cN_c\subset\hat{N},\quad c\,N_c\subset\hat{N},\quad U_c\subset\{z\in\C:|z|<\epsilon\}.
$$
Consider $T'\in U_{N_c,B}$ and $c'\in U_c$. Observe
$$
(c+c')(T+T')=cT+(c'T+cT'+c'T').
$$
For every $b\in B$ we get
$$
(c'T+cT'+c'T')b=c'Tb+cT'b+c'T'b\in c'TB+c\,N_c+U_cN_c\subset \hat{N}+\hat{N}+\hat{N}\subset N.
$$
Therefore, $c'T+cT'+c'T'\in U_{N,B}$. It follows that
$$
(c+U_c)(T+U_{N_c,B})\subset c\,T+U_{N,B},
$$
which yields the desired continuity at $(c,T)$.
\end{proof}

\section{$C^1_F$-maps in Fr\'echet spaces} 

In this section $V, V_1, V_2, F, F_1, F_2$ always denote Fr\'echet spaces. We begin with
a few facts from \cite[Part I]{H} about $C^1_{MB}$-maps
$f:V\supset U\to F$. Each derivative $Df(u):V\to F$, $u\in U$, is linear and continuous. Differentiation of $C^1_{MB}$-maps is linear, and the chain rule holds. We have
\begin{equation}
f(v)-f(u)=\int_0^1Df(u+t(v-u))(v-u)dt\quad\text{for}\quad u+[0,1]v\subset U
\end{equation}
with the Riemann integral of continuous maps $[a,b]\to F$ from \cite[Part I]{H}. Linear continuous maps $T\in L_c(V,F)$  are $C^1_{MB}$-smooth with $DT(u)=T$ everywhere. If $f_1:V\supset U\to F_1$ and $f_2:V\supset U\to F_2$ are $C^1_{MB}$-smooth then also $f_1\times f_2:V\supset U\ni u\mapsto (f_1(u),f_2(u))\in F_1\times F_2$ is $C^1_{MB}$-smooth, with
$$
D(f_1\times f_2)(u)v=(Df_1(u)v,Df_2(u)v).
$$

\begin{prop} (see \cite[Part I]{H}). For continuous $f:V_1\times V_2\supset U\to F$, $U$ open, the following statements are equivalent.\\
(i) For every $(u_1,u_2)\in U$, $v_1\in V_1$, $v_2\in V_2$ the partial derivatives $D_kf(u_1,u_2)v_k$ exist and both maps
$$
U\times V_k\ni(u_1,u_2,v_k)\mapsto D_kf(u_1,u_2)v_k\in F,\quad k\in\{1,2\}, 
$$
are continuous.\\
(ii) $f$ is $C^1_{MB}$-smooth.\\
In this case,
$$
Df(u_1,u_2)(v_1,v_2)=D_1f(u_1,u_2)v_1+D_2f(u_1,u_2)v_2
$$
for all $(u_1,u_2)\in U$, $v_1\in V_1$, $v_2\in V_2$.
\end{prop}

\begin{prop} 
Let $U\subset V$ be open. A map $f:U\to F$ is $C^1_F$-smooth if and only if it is  $C^1_{MB}$-smooth with $U\ni u\mapsto Df(u)\in L_c(V,F)$ $\beta$-continuous..
\end{prop}

\begin{proof} We only show that $C^1_F$-smoothness implies $C^1_{MB}$-smoothness. Assume $f:V\supset U\to F$ is $C^1_F$-smooth. Consider sequences $U\ni u_j\to u\in U$ and $v_j\to v$ in $V$ and let a neighbourhood $N$ of $0$ in $F$ be given. We have to show $Df(u_j)v_j\to Df(u)v$ as $j\to\infty$. The set $B=\{v_j:j\in\N\}$ is bounded. For every $j\in\N$,
$$
Df(u_j)v_j-Df(u)v=[Df(u_j)-Df(u)]v_j+Df(u)[v_j-v].
$$
By continuity of $Df(u)$ the last term tends to $0$ as $j\to\infty$. Now consider the points $[Df(u_j)-Df(u)]v_j\in F$, $j\in\N$. By the $\beta$-continuity of $Df:U\to L_c(V,F)$, $Df(u_j)\in Df(u)+U_{N,B}$ for $j$ sufficiently large. For these $j$,
$$
[Df(u_j)-Df(u)]v_j\in[Df(u_j)-Df(u)]B\subset N.
$$
This yields $[Df(u_j)-Df(u)]v_j\to0$ as $j\to\infty$.
\end{proof}

\begin{prop}
In case $E$ is a finite-dimensional normed space each $C^1_{MB}$-map $f:E\supset U\to F$ is $C^1_F$-smooth.
\end{prop}

\begin{proof}
Recall Remark 2.1 (iii). Let $B\subset E$ be bounded and let $N$ be a neighbourhood of $0$ in $F$. Because of $\dim\,E<\infty$ the closure
$\overline{B}$ is compact. The map $U\times E\ni(y,x)\mapsto Df(y)x\in F$ is continuous. Apply Proposition 1.2 to the compact set $\{u\}\times\overline{B}\subset U\times E$. This yields a neighbourhood $N_u$ of $u$ in $U$ with
$$
(Df(v)-Df(u))b\in N\quad\text{for all}\quad v\in N_u,b\in\overline{B},
$$
hence $Df(v)\in Df(u)+U_{N,B}$ for all $v\in N_u$.
\end{proof}

\begin{prop}
For Banach spaces $V$ and $F$ and $U\subset V$ open a map $f:V\supset U\to F$ is $C^1_F$-smooth if and only if there exists a continuous map $D_f:U\to L_c(V,F)$  such that for every $u\in U$ and\\
(F) $\quad$ for every $\epsilon>0$ there exists $\delta>0$ with
$$
|f(v)-f(u)-D_f(u)(v-u)|\le\epsilon|v-u|\quad\text{for all}\quad v\in U\quad\text{with}\quad |v-u|<\delta.
$$ 
In this case, $D_f(u)v$ is the directional derivative $Df(u)v$, for every $u\in U,v\in V$. 

\end{prop}

\begin{proof}
We only show that for a $C^1_F$-map $f:V\supset U\to F$ and $u\in U$ the map $D_f(u)=Df(u)\in L_c(V,F)$ satisfies statement (F).  Due to Proposition 3.2 we may use the integral representation (3.1) for $C^1_{MB}$-maps. For $v$ in a convex neighbourhood $N\subset U$ of $u$ this yields
$$
|f(v)-f(u)-Df(u)(v-u)|=\left|\int_0^1Df(u+s(v-u))[v-u]ds-\int_0^1Df(u)[v-u]ds\right|
$$
$$
=\left|\int_0^1[ Df(u+s(v-u))-Df(u)][v-u]du\right|\le\int_0^1|\ldots|ds
$$
$$
\le\max_{0\le s\le1}|Df(u+s(v-u))-Df(u)||v-u|.
$$
Now the continuity of $Df$ at $u$ completes the proof.
\end{proof}

Continuous linear maps $T:V\to F$ are $C^1_F$-smooth because they are $C^1_{MB}$-smooth with constant derivative, $DT(u)=T$ for all $u\in V$. Using Proposition 3.2 and continuity of addition and multiplication on $L_c(V,F)$ (with the topology $\beta$) one obtains from the properties of $C^1_{MB}$-maps that linear combinations of $C^1_F$-maps are $C^1_F$-maps, that also for $C^1_F$-maps differentiation is linear, and the integral
formula (3.1) holds.  If $f_1:V\supset U\to F_1$ and $f_2:V\supset U\to F_2$ are $C^1_F$-smooth then also $f_1\times f_2:V\supset U\ni u\mapsto (f_1(u),f_2(u))\in F_1\times F_2$ is $C^1_F$-smooth, with
$$
D(f_1\times f_2)(u)v=(Df_1(u)v,Df_2(u)v).
$$
This follows easily from the analogous property for $C^1_{MB}$-maps, by means of the formula for the directional derivatives of $f_1\times f_2$ and considering neighbourhoods of $0$ in $F_1\times F_2$ which are products of neighbourhoods of $0$ in $F_j$, $j\in\{1,2\}$. 

\begin{prop} 
(Chain rule). If $f:V\supset U\to F$ and $g:F\supset W\to G$ are $C^1_F$-maps, with $f(U)\subset W$, then also
$g\circ f$ is a $C^1_F$-map.
\end{prop}

\begin{proof} 1. The chain rule for $C^1_{MB}$-maps yields that $g\circ f$ is $C^1_{MB}$-smooth with $D(g\circ f)(u)=D(g(f(u)))\circ Df(u)$ for all $u\in U$.   So it remains to prove that the map $U\ni u\mapsto Dg(f(u))\circ Df(u)\in L_c(V,G)$ is $\beta$-continuous. As $V$ has countable neighbourhood bases it is enough to show that, given a sequence $U\ni u_j\to u\in U$, a bounded set $B\subset V$, and a neighbourhood $N$ of $0$ in $G$,  we have
$$
[Dg(f(u_j))\circ Df(u_j)-Dg(f(u))\circ Df(u)]B\subset N\quad\text{for}\quad j\in\N\quad\text{sufficiently large.}
$$
So let a sequence $U\ni u_j\to u\in U$, a bounded set $B\subset V$, and a neighbourhood $N$ of $0$ in $G$ be given.

\medskip

2. There is a neighbourhood $N_1$ of $0$ in $G$ with $N_1+N_1+N_1+N_1\subset N$, see \cite[proof of Theorem 1.10]{R}.
By linearity, for every $j\in\N$,
\begin{eqnarray*}
Dg(f(u_j))\circ Df(u_j)-Dg(f(u))\circ Df(u) & = & [Dg(f(u_j))-Dg(f(u))]\circ Df(u_j))\\
& & +Dg(f(u))\circ[Df(u_j)-Df(u)].
\end{eqnarray*}
3. Consider the last term. By continuity of $Dg(f(u))$ at $0\in F$, there is a neighbourhood $N_2$ of $0$ in $F$ with
$Dg(f(u))N_2\subset N_1$. By $\beta$-continuity of $Df$ at $u\in U$, there is an integer $j_s$ such that for all integers $j\ge j_2$,
$$
[Df(u_j)-Df(u)]B\subset N_2.
$$
Hence, for all integers $j\ge j_2$,
$$
Dg(f(u))\circ[Df(u_j)-Df(u)]B\subset N_1.
$$
4. $Df(u)B$ is bounded. Using $\beta$-continuity of $Dg$ at $f(u)$ and $\lim_{j\to\infty}f(u_j)=f(u)$ we find an integer $j_3\ge j_2$ such that for all integers $j\ge j_3$
we have 
$$
[Dg(f(u_j))-Dg(f(u))]Df(u)B\subset N_1.
$$
5. Now we use the continuity of $W\times F\ni(w,h)\mapsto Dg(w)h\in G$ at $(f(u),0)$. We find a neighbourhood $N_3$ of $0$ in $F$ and an integer $j_4\ge j_3$ such that for all integers $j\ge j_4$ we have $Dg(f(u_j))N_3\subset N_1$ and 
$-Dg(f(u))N_3\subset N_1$. This yields
$$
[Dg(f(u_j))-Dg(f(u))]N_3\subset N_1+N_1\quad\text{for all integers}\quad j\ge j_4.
$$
6. The $\beta$-continuity of $Df$ at $u\in U$ yields an integer $j_N\ge j_4$ such that for all integers $j\ge j_N$ we have
$$
[Df(u_j)-Df(u)]B\subset N_3,
$$
hence
$$
Df(u_j)B\subset Df(u)B+N_3.
$$
7. For integers $j\ge j_N$ we obtain
$$
[Dg(f(u_j))\circ Df(u_j)-Dg(f(u))\circ Df(u)]B
$$
$$
=[Dg(f(u_j))-Dg(f(u))]\circ Df(u_j))+Dg(f(u))\circ[Df(u_j)-Df(u)]B\quad\text{(see part 2)}
$$
$$
\subset [Dg(f(u_j))-Dg(f(u))] Df(u_j))B+[Dg(f(u))\circ[Df(u_j)-Df(u)]B
$$
$$
\subset[Dg(f(u_j))-Dg(f(u))](Df(u)B+N_3)+N_1\quad\text{(see parts 6 and 3)}
$$
$$
\subset[Dg(f(u_j))-Dg(f(u))]Df(u)B+[Dg(f(u_j))-Dg(f(u))]N_3+N_1
$$
$$
\subset N_1+(N_1+N_1)+N_1\quad\text{(see parts 4 and 5)}
$$
$$
\subset N.
$$
\end{proof}

\begin{prop} For a continuous map $f:V_1\times V_2\supset U\to F$, $U$ open, the following statements are equivalent.\\
(i) For all $(u_1,u_2)\in U$ and all $v_k\in V_k$, $k\in\{1,2\}$, $f$ has a partial derivative $D_kf(u_1,u_2)v_k\in F$, all maps
$$
D_kf(u_1,u_2):V_k\to F,\quad(u_1,u_2)\in U,\quad k\in\{1,2\},
$$
are linear and continuous, and the maps
$$
U\ni(u_1,u_2)\mapsto D_kf(u_1,u_2)\in L_c(V_k,F),\quad k\in\{1,2\},
$$ 
are $\beta$-continuous.\\
(ii) $f$ is $C^1_F$-smooth.\\
In this case,
$$
Df(u_1,u_2)(v_1,v_2)=D_1f(u_1,u_2)v_1+D_2f(u_1,u_2)v_2
$$
for all $(u_1,u_2)\in U$, $v_1\in V_1$, $v_2\in V_2$.
\end{prop}

\begin{proof} 1. Suppose (ii) holds. Then $f$ is $C^1_{MB}$-smooth, and all statements in (i) up to the last one follow from Proposition 3.1 on partial derivatives. In order to deduce the last statement in (i) for $k=1$ let a sequence $((u_{1j},u_{2j}))_{j=1}^{\infty}$ in $U$ be given which converges to some $(u_1,u_2)\in U$. Let a neighbourhood $N$ of $0 $ in $F$ and a bounded set $B_1\subset V_1$ be given, consider the neighbourhood $U_{N,B_1}$ of $0$ in $L_c(V_1,F)$. As $V_1\ni v\mapsto(v,0)\in V_1\times V_2$ is linear and continuous, $B_1\times\{0\}$
is a bounded subset of $V_1\times V_2$. As $f$ is $C^1_F$-smooth the map $Df$ is $\beta$-continuous, and for $j$ sufficiently large we get
$(Df(u_{1j},u_{2j})-Df(u_1,u_2))(B_1\times\{0\})\subset N$ which yields
$(D_1f(u_{1j},u_{2j})-D_1f(u_1,u_2))B_1\subset N$.
For $k=2$ the proof is analogous.

\medskip

2. Suppose (i) holds.

\medskip

2.1. Claim: Both maps $U\times V_k\ni(u_1,u_2,v_k)\mapsto D_kf(u_1,u_2)v_k\in F$, $k\in\{1,2\}$, are continuous.

\medskip

Proof for $k=1$: Let a sequence $(u_{1j},u_{2j},v_{1j})_1^{\infty}$ in $U\times V_1$ be given which converges to some $(u_1,u_2,v_1)\in U\times V_1$. Then $v_{1j}\to v_1$ in $V_1$, and $B_1=\{v_{1j}:j\in\N\}\cup\{v_1\}$ is a bounded subset of $V_1$. Let a neighbourhood $N$ of $0$ in $F$ be given. By the $\beta$-continuity of $D_1f$,
$$
(D_1f(u_{1j},u_{2j})-D_1f(u_1,u_2))B_1\subset N\quad\text{for}\quad j\quad\text{sufficiently large}.
$$
For each $j\in N$ we have
\begin{eqnarray*}
D_1f(u_{1j},u_{2j})v_{1j}-D_1f(u_1,u_2)v_1 & = & (D_1f(u_{1j},u_{2j})-D_1f(u_1,u_2))v_{1j}\\
& & + D_1f(u_1,u_2))(v_{1j}-v_1).
\end{eqnarray*}
Now it becomes obvious how to complete the proof, using the last equation, the statement right before it, and continuity of 
$D_1f(u_1,u_2)$.

\medskip

2.2. Proposition 3.1 on partial derivatives applies and yields that $f$ is $C^1_{MB}$-smooth, with
$$
Df(u_1,u_2)(v_1,v_2)=D_1f(u_1,u_2)v_1+D_2f(u_1,u_2)v_2
$$
for all $(u_1,u_2)\in U$, $v_1\in V_1$, $v_2\in V_2$. According to Proposition 3.2 it remains to prove that the map $Df:U\to L_c(V_1\times V_2,F)$
is $\beta$-continuous. The projections $pr_k$ of $V_1\times V_2$
onto the factor $V_k$, for $k\in\{1,2\}$, are linear and continuous. For every $(u_1,u_2)\in U$ we have
$$
Df(u_1,u_2)=D_1f(u_1,u_2)\circ pr_1+D_2f(u_1,u_2)\circ pr_2,
$$
so it is sufficient to show that both maps
$$
V_1\times V_2\supset U\ni(u_1,u_2)\mapsto D_kf(u_1,u_2)\circ pr_k\in L_c(V_1\times V_2,F),\quad k\in\{1,2\},
$$
are $\beta$-continuous. We deduce this for $k=1$. Let a sequence $(u_{1j},u_{2j})_1^{\infty}$ in $U$ be given which converges to some $(u_1,u_2)\in U$, as well as a bounded subset $B\subset V_1\times V_2$ and a neighbourhood $N$ of $0$ in $F$. We need to show 
$$
(D_1f(u_{1j},u_{2j})\circ pr_1-D_1f(u_1,u_2)\circ pr_1)B\subset N
$$
for $j\in\N$ sufficiently large. $B_1=pr_1B$ is a bounded subset of $V_1$, and for every $j\in\N$ we have
$$
(D_1f(u_{1j},u_{2j})\circ pr_1-D_1f(u_1,u_2)\circ pr_1)B\subset (D_1f(u_{1j},u_{2j})-D_1f(u_1,u_2))B_1.
$$
The $\beta$-continuity of the map $D_1f$ yields that the last set is contained in $N$ for $j$ sufficiently large.
\end{proof}

\section{$C^1_F$-submanifolds}

\medskip

$C^1_F$-submanifolds of a Fr\'echet space are defined in the same way as continuously differentiable submanifolds of a Banach space. Below we collect the simple facts which are used in Section 7 and in Parts II and III.

\medskip

A $C^1_F$-diffeomorphism is an injective $C^1_F$-map from an open subset $U$ of a Fr\'echet space $F$ onto an open subset $W$ of a Fr\'echet space $V$ whose inverse defined on $W\subset V$ is a $C^1_F$-map. 

\medskip

Let $F=G\oplus H$ be a direct sum decomposition of a Fr\'echet space $F$ into closed subspaces. A subset $M\subset F$ is a $C^1_F$-submanifold of $F$ (modelled over the Fr\'echet space $G$) if for every point $m\in M$ there are an open neighbourhood $U$ in $F$ and a $C^1_F$-diffeomorphism $K:U\to F$ onto $W=K(U)$ with
$$
K(M\cap U)=W\cap G.
$$

\medskip

The tangent cones of the $C^1_F$-submanifold $M$ are closed subspaces of $F$. For $K$ as before the map $(DK(m))^{-1}$ defines a topological isomorphism from $G$ onto $T_mM$, and $K^{-1}$ defines an injective map $P$ from the open neighbourhood $V\cap G$ of $K(m)$ in $G$ onto the open neighbourhood $U\cap M$ of $m$ in $M$.   

\medskip

Open subsets of $C^1_F$-submanifolds are $C^1_F$-submanifolds.

\medskip

A $C^1_F$-map $h:M\to H$, $M$ a $C^1_F$-submanifold of $F$  and $H$ a Fr\'echet space, is defined by the property that for all local parametrizations $P$ as above the composition $f\circ P$ is a $C^1_F$-map.

\medskip

For $h$ as before and $m\in M$ the derivative $T_mh:T_mM\to H$ is defined by $T_mh(t)=(h\circ c)'(0)$,
for any continuously differentiable curve $c:I\to F$ with $c(0)=m$, $c(I)\subset M$, $c'(0)=t$. The map $T_mh$ is
linear and continuous.

\medskip

In case $h(M)$ is contained in a $C^1_F$-submanifold $M_H$ of $H$ and $z:M_H\to 
Z$ is $C^1_F$-smooth the chain rule holds, with $T_mh(T_mM)\subset T_{h(m)}M_H$ and $T_m(z\circ h)t=T_{h(m)}zT_mh(t)$.


\medskip

The restriction of a $C^1_F$-map on an open subset of $F$ to a  $C^1_F$-submanifold $M$ of $F$, with range in a Fr\'echet space $H$,  is a $C^1_F$-map from $M$ into the target space. 

\medskip

\section{Uniform contractions}

\medskip

The proof of Theorem 5.2 below employs twice the following basic uniform contraction principle.

\begin{prop}
(See for example \cite[Appendix VI, Proposition 1.2]{DvGVLW}.) Let a Hausdorff space $T$, a complete metric space $M$, and a map  $f:T\times M\to M$ be given. Assume that $f$ is a uniform contraction in the sense that there exists $k\in[0,1)$ so that 
$$
d(f(t,x),f(t,y))\le k\,d(x,y)
$$
for all $t\in T,x\in M,y\in M$, and $f(\cdot,x):T\to M$ is continuous for each $x\in M$. Then the map $g:T\to M$ given by $g(t)=f(t,g(t))$ is continuous.
\end{prop}

\begin{thm}
Let a Fr\'echet space $T$, a Banach space $B$, open sets $V\subset F$ and $O_B\subset B$, and a $C^1_F$-map $A:V\times O_B\to B$ be given. Assume that for a closed set $M\subset O_B$ we have $A(V\times M)\subset M$, and $A$ is a uniform contraction  in the sense that there exists $k\in[0,1)$ so that 
$$
|A(t,x)-A(t,y)|\le k|x-y|
$$
for all $t\in V,x\in O_B,y\in O_B$. Then the map $g:V\to B$ given by $g(t)=A(t,g(t))\in M$ is $C^1_F$-smooth.
\end{thm}

Notice that the derivative $\Gamma=Dg(t)\hat{t}$ of the map $g$  satisfies the equation
\begin{equation}
\Gamma=D_1A(t,g(t))\hat{t}+D_2A(t,g(t))\Gamma.
\end{equation}

{\it Proof} of Theorem 5.2.
1. $A$ is continuous. So Proposition 5.1 applies to the restriction of $A$ to $V\times M$ and yields a continuous map $g:V\to B$ with $g(t)=A(t,g(t))\in M$ for all $t\in V$. Choose
$\kappa\in(k,1)$. Each linear map $D_2A(t,x):B\to B$, $(t,x)\in V\times O_B$, is continuous. The contraction property yields
$$
|D_2A(t,x)|=\sup_{|\hat{x}|\le1}|D_2A(t,x)\hat{x}|\le\kappa\quad\text{for all}\quad(t,x)\in V\times O_B
$$
since given $\epsilon=\kappa-k$ and $t\in V$, $x\in O_B$, and $\hat{x}\in B$ with $|\hat{x}|\le1$ there exists $\delta>0$ such that for $h=\frac{\delta}{2}$,
$$
x+h\hat{x}\in O_B\quad\text{and}
$$
\begin{eqnarray*}
|h^{-1}(A(t,x)-A(t,x+h\hat{x})) & - &  D_2A(t,x)\hat{x}|\\
=\quad|h^{-1}(A(t,x)-A(t,x+h\hat{x})) & - & DA(t,x)(0,\hat{x})|\quad\le\quad\epsilon,
\end{eqnarray*}
hence
\begin{eqnarray*}
|h||D_2A(t,x)\hat{x}| & \le & \epsilon|h|+|A(t,x+h\hat{x})-A(t,x)|\\
& \le & \epsilon|h|+k|h\hat{x}|\le(\epsilon+k)|h|=\kappa|h|.
\end{eqnarray*}
Divide by $|h|=h$.

\medskip

2. It follows that each map $id_B-D_2A(t,x)\in L_c(B,B)$, $t\in V$ and $x\in O_B$, is a topological isomorphism. As
$A$ is $C^1_F$-smooth we get that the map
$$
V\times O_B\ni(t,x)\mapsto D_2A(t,x)\in L_c(B,B)
$$
is $\beta$-continuous, or equivalently, continuous with respect to the usual norm-topology on $L_c(B,B)$. As inversion is continuous we see that also the map
$$
V\times O_B\ni(t,x)\mapsto(id_B-D_2A(t,x))^{-1}\in L_c(B,B)
$$
is continuous.

\medskip

3. For all $(t,x,\hat{t})\in V\times O_B\times T$ and for all $\hat{x},\hat{y}$ in $B$ we have
\begin{eqnarray*}
|DA(t,x)(\hat{t},\hat{x})-DA(t,x)(\hat{t},\hat{y})| & = & |DA(t,x)(0,\hat{x}-\hat{y})|\\
& = & |D_2A(t,x)(\hat{x}-\hat{y})|\le\kappa|\hat{x}-\hat{y}|.
\end{eqnarray*}
Hence Proposition 5.1 applies to the version 
$$
\Gamma=D_1A(t,x)\hat{t}+D_2A(t,x)\Gamma
$$
of Eq. (5.1) with parameters $(t,x,\hat{t})\in V\times O_B\times T$ and yields a continuous map $\gamma:V\times O_B\times T\to B$
with
$$
\gamma(t,x,\hat{t})=D_1A(t,x)\hat{t}+D_2A(t,x)\gamma(t,x,\hat{t})\quad\text{for all}\quad(t,x,\hat{t})\in V\times O_B\times T,
$$
or equivalently,
$$
\gamma(t,x,\hat{t})=(id_B-D_2A(t,x))^{-1}D_1A(t,x)\hat{t}\quad\text{for all}\quad(t,x,\hat{t})\in V\times O_B\times T.
$$
This shows that each map $\gamma(t,x,\cdot)$, $(t,x)\in V\times O_B$, belongs to $ L_c(T,B)$.

\medskip

Claim: The map
$$
\tilde{\gamma}:V\times O_B\ni(t,x)\mapsto\gamma(t,x,\cdot)\in L_c(T,B)
$$
is $\beta$-continuous.

\medskip

Proof. Let a sequence $(t_j,x_j)_1^{\infty}$ in $V\times O_B$ converge to a point $(t,x)\in V\times O_B$. Consider a neighbourhood $N$ of $0$ in $B$ and a bounded set $T_b\subset T$. We have to show that for $j\in\N$ suffiently large, $(\tilde{\gamma}(t_j,x_j)-\tilde{\gamma}(t,x))T_b\subset N$. For all $j\in\N$ and all $\hat{t}\in T_b$ we have
$$
|(\tilde{\gamma}(t_j,x_j)-\tilde{\gamma}(t,x))\hat{t}|=|((id_B-D_2A(t_j,x_j))^{-1}D_1A(t_j,x_j)
$$
$$
-(id_B-D_2A(t,x))^{-1}D_1A(t,x))\hat{t}|
$$
$$
\le|(id_B-D_2A(t_j,x_j))^{-1}-(id_B-D_2A(t,x))^{-1})||D_1A(t_j,x_j)\hat{t}|
$$
$$
+|(id_B-D_2A(t,x))^{-1})||(D_1A(t_j,x_j)-D_1A(t,x))\hat{t}|
$$
$$
\le|(id_B-D_2A(t_j,x_j))^{-1}-(id_B-D_2A(t,x))^{-1})|(|(D_1A(t_j,x_j)-D_1A(t,x))\hat{t}|
$$
$$
+|D_1A(t,x)\hat{t}|)+|(id_B-D_2A(t,x))^{-1})||(D_1A(t_j,x_j)-D_1A(t,x))\hat{t}|.
$$
Now it becomes obvious how to complete the proof, using 
$$
|(id_B-D_2A(t_j,x_j))^{-1}-(id_B-D_2A(t,x))^{-1})|\to0\quad\text{as}\quad j\to\infty,
$$
boundedness of $|D_1A(t,x)T_b|$, and $\beta$-continuity of the partial derivative
$$
D_1A:V\times O_B\to L_c(T,B)
$$
due to Proposition 3.6. 

\medskip

4. Consider the continuous map $\xi:V\times T\ni(t,\hat{t})\mapsto\gamma(t,g(t),\hat{t})\in B$. Using part 3 we observe that the map
$V\ni t\mapsto\xi(t,\cdot)\in L_c(T,B)$ is $\beta$-continuous. It remains to show that for all
$t\in V$ and all $\hat{t}\in T$ we have
$$
\lim_{0\neq h\to0}\frac{1}{h}(g(t+h\hat{t})-g(t))=\xi(t,\hat{t}),
$$
which means that the directional derivative $Dg(t)\hat{t}$ exists and equals $\xi(t,\hat{t})$.

\medskip

So let $t\in V$ and $\hat{t}\in T$ be given. Choose a convex neighbourhood $N_B\subset O_B$ of $g(t)$. There exists $\delta>0$ such that for $-\delta\le h\le\delta$,
$$
t+h\hat{t}\in V\quad\text{and}\quad g(t+h\hat{t})\in N_B.
$$ 
Notice that for all $h\in[-\delta,\delta]$ and for all $\theta\in[0,1]$,
$$
g(t)+\theta(g(t+h\hat{t})-g(t))\in N_B.
$$
With the abbreviation 
\begin{eqnarray*}
\xi=\xi(t,\hat{t}) & = & \gamma(t,g(t),\hat{t})=D_1A(t,g(t))\hat{t}+D_2A(t,g(t))\gamma(t,g(t),\hat{t})\\
& = & D_1A(t,g(t))\hat{t}+D_2A(t,g(t))\xi
\end{eqnarray*} 
one finds that
$$
h^{-1}(g(t+h\hat{t})-g(t))-\xi=h^{-1}(A(t+h\hat{t},g(t+h\hat{t}))-A(t,g(t))-\xi,\quad\text{with}\quad 0<|h|<\delta,
$$
equals
\begin{eqnarray*}
= & \, & h^{-1}(A(t+h\hat{t},g(t+h\hat{t}))-A(t+h\hat{t},g(t))-D_1A(t,g(t))\hat{t}-D_2A(t,g(t))\xi\\
& & + h^{-1}(A(t+h\hat{t},g(t))-A(t,g(t)))\\
= & \, & h^{-1}(A(t+h\hat{t},g(t))-A(t,g(t)))-D_1A(t,g(t))\hat{t}\\
& & + h^{-1}(A(t+h\hat{t},g(t+h\hat{t}))-A(t+h\hat{t},g(t))\\
& & -\int_0^1D_2A(t+h\hat{t},g(t)+\theta[g(t+h\hat{t})-g(t)])\xi d\theta\\
& & +\int_0^1\{D_2A(t+h\hat{t},g(t)+\theta[g(t+h\hat{t})-g(t)])-D_2A(t,g(t))\}\xi d\theta\\
= & \, &  h^{-1}(A(t+h\hat{t},g(t))-A(t,g(t)))-D_1A(t,g(t))\hat{t}\\
& & +\int_0^1h^{-1}D_2A(t+h\hat{t},g(t)+\theta[g(t+h\hat{t})-g(t)])[g(t+h\hat{t})-g(t)]d\theta\\
& & - \int_0^1D_2A(t+h\hat{t},g(t)+\theta[g(t+h\hat{t})-g(t)])\xi d\theta\\
& & + \int_0^1\{D_2A(t+h\hat{t},g(t)+\theta[g(t+h\hat{t})-g(t)])-D_2A(t,g(t))\}\xi d\theta\\
= & \, &  h^{-1}(A(t+h\hat{t},g(t))-A(t,g(t)))-D_1A(t,g(t))\hat{t}\\
& & + \int_0^1D_2A(t+h\hat{t},g(t)+\theta[g(t+h\hat{t})-g(t)])[h^{-1}(g(t+h\hat{t})-g(t))-\xi]d\theta\\
& &  + \int_0^1\{D_2A(t+h\hat{t},g(t)+\theta[g(t+h\hat{t})-g(t)])-D_2A(t,g(t))\}\xi d\theta.
\end{eqnarray*}
Hence
$$
|h^{-1}(g(t+h\hat{t})-g(t))-\xi|
$$
is majorized by
\begin{eqnarray*}
\, & \, &|h^{-1}(A(t+h\hat{t},g(t))-A(t,g(t)))-D_1A(t,g(t))\hat{t}|\\
& & +\kappa|h^{-1}(g(t+h\hat{t})-g(t))-\xi|\\
& & + \left|\int_0^1\{D_2A(t+h\hat{t},g(t)+\theta[g(t+h\hat{t})-g(t)])-D_2A(t,g(t))\}\xi d\theta\right|,
\end{eqnarray*}
which yiekds
\begin{eqnarray*}
\, & \, & (1-\kappa)
|h^{-1}(g(t+h\hat{t})-g(t))-\xi|\\
& \le &  |h^{-1}(A(t+h\hat{t},g(t))-A(t,g(t)))-D_1A(t,g(t))\hat{t}|\\
& & + \left|\int_0^1\{D_2A(t+h\hat{t},g(t)+\theta[g(t+h\hat{t})-g(t)])-D_2A(t,g(t))\}\xi d\theta\right|.
\end{eqnarray*}
The first term in the last expression converges to $0$ as $0\neq h\to0$. The map
$$
[-\delta,\delta]\times[0,1]\ni(h,\theta)\mapsto\{D_2A(t+h\hat{t},g(t)+\theta[g(t+h\hat{t})-g(t)])-D_2A(t,g(t))\}\xi\in B
$$
is uniformly continuous with value $0$ on $\{0\}\times[0,1]$. This implies that for $0\neq h\to0$ the last integrand converges to $0$ uniformly with respect to $\theta\in[0,1]$. Therefore the last integral tends to $0$ as $0\neq h\to0$.
$\Box$

\section{An implicit function theorem}

From Theorem 5.2 one obtains the following Implicit Function Theorem in the usual way, paying attention to  
$C^1_F$-smoothness.

\begin{thm}
Let a Fr\'echet space $T$, Banach spaces $B$ and $E$, an open set $U\subset T\times B$, a $C^1_F$-map $f:U\to E$, and a zero $(t_0,x_0)\in U$ of $f$ be given. Assume that $D_2f(t_0,x_0):B\to E$ is bijective.
Then there are open neighbourhoods $V$ of $t_0$ in $T$ and $W$ of $x_0$ in $B$ with $V\times W\subset U$ and a $C^1_F$-map $g:V\to W$ with
$g(t_0)=x_0$ and
$$
\{(t,x)\in V\times W:f(t,x)=0\}=\{(t,x)\in V\times W:x=g(t)\}.
$$
\end{thm}

\begin{proof}
1. (A fixed point problem) Choose an open neighbourhood $N_{T,1}$ of $t_0$ and a convex open neighbourhood $N_B$ of $x_0$ in $B$ with $N_{T,1}\times N_B\subset U$. The equation
$$
f(t,x)=f(t,x_0)+D_2f(t_0,x_0)[x-x_0]+R(t,x)
$$
defines a $C^1_F$-map $R:N_{T,1}\times N_B\to E$, with $R(t,x_0)=0$ for all $t\in N_{T,1}$,
$$
D_2R(t,x)=D_2f(t,x)-D_2f(t_0,x_0)\quad\text{for all}\quad t\in N_{T,1}\quad\text{and}\quad x\in N_B,
$$
and in particular, $D_2R(t_0,x_0)=0$. The map
$$
N_{T,1}\times N_B\ni(t,x)\mapsto D_2R(t,x)\in L_c(B,E)
$$
is $\beta$-continuous. In order to solve the equation $0=f(t,x)$, $(t,x)\in N_{T,1}\times N_B$, for $x$ as a function of $t$,  observe that this equation is equivalent to
$$
0=f(t,x_0)+D_2f(t_0,x_0)[x-x_0]+R(t,x),\\
$$
or,
\begin{eqnarray*}
x & = & x_0+(D_2f(t_0,x_0))^{-1}[-f(t,x_0)-R(t,x)]\\
& = & x_0-(D_2f(t_0,x_0))^{-1}f(t,x_0)-(D_2f(t_0,x_0))^{-1}R(t,x).
\end{eqnarray*}
The last expression defines a  map
$$
A:N_{T,1}\times N_B\to B
$$
with $A(t_0,x_0)=x_0$, and for $(t,x)\in N_{T,1}\times N_B$, 
$$
0=f(t,x)\quad\text{if and only if}\quad x=A(t,x).
$$ 
The map $A$ is $C^1_F$-smooth  since the linear map $D_2f(t_0,x_0))^{-1}:E\to B$ is continuous, due to the open mapping theorem.

2. (Contraction) For all $t\in N_{T,1}$ and for all $x,\hat{x}$ in $N_B$,
\begin{eqnarray*}
|A(t,\hat{x})-A(t,x)| & = & |-(D_2f(t_0,x_0))^{-1}R(t,\hat{x})+(D_2f(t_0,x_0))^{-1}R(t,x)|\\
& \le & |(D_2f(t_0,x_0))^{-1}|\left|\int_0^1D_2R(t,x+s[\hat{x}-x])[\hat{x}-x]ds\right|.
\end{eqnarray*}
Let 
$$
\epsilon=\frac{1}{2|(D_2f(t_0,x_0))^{-1}|}.
$$
There are an open  neighbourhood $N_{T,2}\subset N_{T,1}$ of $t_0$ and $\delta>0$ such that for
all $t\in N_{T,2}$ and all $x\in B$ with $|x-x_0|\le\delta$,
$$
x\in N_B\quad\text{and}\quad|D_2R(t,x)|=|D_2R(t,x)-D_2R(t_0,x_0)|<\epsilon.
$$
For all $x\neq\hat{x}$ in $B$ with $|x-x_0|\le\delta$ and $|\hat{x}-x_0|\le\delta$ and for all $t\in N_{T,2}$ and  $s\in[0,1]$ it follows that
$|x+s[\hat{x}-x]-x_0|\le\delta$, hence
$$
\left|D_2R(t,x+s[\hat{x}-x])\frac{1}{|\hat{x}-x|}[\hat{x}-x]\right|<\epsilon,
$$
and thereby 
$$
|A(t,\hat{x})-A(t,x)|\le\epsilon|\hat{x}-x||(D_2f(t_0,x_0))^{-1}|=\frac{1}{2}|\hat{x}-x|.
$$

3. (Invariance) By continuity there is an open neighbourhood $N_{T,3}\subset N_{T,2}$ of $t_0$ such that
$$
|A(t,x_0)-A(t_0,x_0)|<\frac{\delta}{4}\quad\text{for all}\quad t\in N_{T,3}.
$$
For all $t\in N_{T,3}$ and $x\in B$ with $|x-x_0|\le\delta$ this yields
\begin{eqnarray*}
|A(t,x)-x_0| & = & |A(t,x)-A(t_0,x_0)|\le|A(t,x)-A(t,x_0)|+|A(t,x_0)-A(t_0,x_0)|\\
& < & \frac{1}{2}|x-x_0|+\frac{\delta}{4}\le\frac{\delta}{2}+\frac{\delta}{4}=\frac{3\delta}{4}.
\end{eqnarray*}

4. Set $V=N_{T,3}$, $O_B=\{x\in B:|x-x_0|<\delta\}$, and 
$$
M=\left\{x\in B:|x-x_0|\le\frac{3\delta}{4}\right\},
$$ 
and apply Theorem 5.2 to the restriction of $A$ to the set $V\times O_B$. This yields a $C^1_F$-map $g:V\to B$ with
$g(t)=A(t,g(t))\in O_B$ for all $t\in V$. Using Part 3 we get $|g(t)-x_0|<\frac{3\delta}{4}$ for all $t\in V$. Set 
$$
W=\left\{x\in B:|x-x_0|<\frac{3\delta}{4}\right\}.
$$ 
Then $g(V)\subset W$. From $g(t)=A(t,g(t))$ for all $t\in V$ we obtain $0=f(t,g(t))$ for these $t$. Conversely, if $0=f(t,x)$ for $(t,x)\in V\times W\subset V\times M$, then $x=A(t,x)$, hence $x=g(t)$. In particular,
$x_0=g(t_0)$.
\end{proof}

\section{Submanifolds by transversality and embedding}

\begin{prop}
Let a $C^1_F$-map $g: F\supset U\to G$ and a $C^1_F$-submanifold $M\subset G$ of finite codimension $m$ be given. Assume that $g$ and $M$ are transversal at a point $x\in g^{-1}(M)$ in the sense that
$$
G=Dg(x)F+T_{g(x)}M.
$$
Then there is an open neighbourhood $V$ of $x$ in $U$ so that $V\cap g^{-1}(M)$ is a $C^1_F$-submanifold of codimension $m$ in $F$, and $T_x(g^{-1}(M)\cap V)=Dg(x)^{-1}T_{g(x)}M$.

\medskip

In case $\dim\,G=m$, $M=\{g(x)\}$, and $Dg(x)$ surjective the assertion holds with $T_{g(x)}M=\{0\}$
\end{prop}

{\bf Proof} for $M\neq\{g(x)\}$.

\medskip

1. There are an open neighbourhood $N_G$ of $\gamma=g(x)$ in $G$ and a $C^1_F$-diffeomorphism $K:N_G\to G$ onto an open set $U_G\subset G$ such that $K(\gamma)=0$, $K(N_G\cap M)=U_G\cap T_{\gamma}M$. We may assume $DK(\gamma)=id$ since otherwise we can replace $K$ with $DK(\gamma)^{-1}\circ K$. Then $DK(\gamma)=id$ maps $T_{\gamma}M$ onto itself.

\medskip

2. By transversality and codim $M=m$ we find a subspace $Q\subset Dg(x)F$ of dimension $m$ which complements $T_{\gamma}M$ in $G$,
$$
G=T_{\gamma}M\oplus Q.
$$ 
The projection $P:G\to Q$ along $T_{\gamma}M$ onto $Q$ is linear and continuous (see \cite[Theorem 5.16]{R}),
and
$PDK(\gamma)Dg(x)=PDg(x)$ is surjective. The preimage $U_F=g^{-1}(N_G)$ is open, with $x\in U_F\subset U$.
For $z\in U_F$ we have
$$
z\in g^{-1}(M)\cap U_F\Leftrightarrow g(z)\in M\cap N_G\Leftrightarrow PK(g(z))=0.
$$
For the $C^1_F$-map $h=P\circ K\circ(g|_{U_F})$ we infer $g^{-1}(M)\cap U_F=h^{-1}(0)$. The derivative
$Dh(x):F\to Q$ is surjective. It follows that there is a subspace $R$ of $F$ with $\dim\, R=\dim\,Q=m$ and
$$
F=Dh(x)^{-1}(0)\oplus R.
$$
The restriction $Dh(x)|_R$ is an isomorphism.

\medskip

3. The $C^1_F$-map
$$
H:\{(z,r)\in Dh(x)^{-1}(0)\times R:x+z+r\in U_F\}\ni(z,r)\mapsto h(x+z+r)\in Q
$$
satisfies $H(0,0)=0$. Because of $D_2H(0,0)\hat{r}=Dh(x)\hat{r}$ for all $\hat{r}\in R$ and $\dim\, R=\dim\,Q$ the map $D_2H(0,0)$ is an isomorphism. Theorem 6.1  yields convex open neighbourhoods $V_H$ of $0$ in $Dh(x)^{-1}(0)$ and $V_R$ of $0$ in $R$, with $x+V_H+V_R\subset U_F$, and a $C^1_F$-map $w:V_H\to V_R$ with $w(0)=0$ and
$$
(V_H\times V_R)\cap H^{-1}(0)=\{(z,r)\in V_H\times V_R:r=w(z)\}.
$$
For every $y\in x+V_H+V_R$, $y=x+z+r$ with $z\in V_H$ and $r\in V_R$, we have
$$
y\in g^{-1}(M)\cap U_F\Leftrightarrow h(y)=0\Leftrightarrow h(x+z+r)=0\Leftrightarrow H(z,r)=0\Leftrightarrow r=w(z).
$$
Hence $g^{-1}(M)\cap(x+V_H+V_R)=\{x+z+w(z):z\in V_H\}$, which implies that $g^{-1}(M)\cap(x+V_H+V_R)$ is a  $C^1_F$-submanifold of $F$, with codimension equal to $\dim\,R=\dim\,Q=m$. Set $V=x+V_H+V_R$.

\medskip

4. (On tangent spaces) From $g^{-1}(M)\cap U_F=h^{-1}(0)$ and $h(x)=0$ we get $h(g^{-1}(M)\cap V)=\{0\}$, hence
$Dh(x)T_x(g^{-1}(M)\cap V)=\{0\}$, or
$$
T_x(g^{-1}(M)\cap V)\subset Dh(x)^{-1}(0).
$$
As both spaces have the same codimension $m$ they are equal. For every $v\in F$ we have
\begin{eqnarray*}
v\in Dh(x)^{-1}(0) & \Leftrightarrow &  Dh(x)v=0\Leftrightarrow PDg(x)v=0\\
&  \Leftrightarrow &  Dg(x)v\in P^{-1}(0)=T_{g(x)}M\Leftrightarrow v\in Dg(x)^{-1}T_xM.
\end{eqnarray*}
Using this we obtain
$$
T_x(g^{-1}(M)\cap V)=Dh(x)^{-1}(0)=Dg(x)^{-1}T_xM.\qquad \Box
$$

\begin{prop}
Suppose $W$ is an open subset of a finite-dimensional normed space $V$, $b\in W$, $F$ is a Fr\'echet space, 
$j:V\supset W\to F$ is a $C^1_F$-map, and $Dj(b)$ is injective.
Then there is an open neighbourhood $N$ of $j(b)$ in $F$ such that $N\cap j(W)$ is a $C^1_F$-submanifold of $F$, with $T_{j(b)}(N\cap j(W))=Dj(b)V$ (hence $\dim\,(N\cap j(W))=\dim\,V$).
\end{prop}

\begin{proof} 
1.The topology induced  by $F$ on the finite-dimensional subspace $Y=Dj(b)V$ of $F$ is given by a norm \cite[Section 1.19]{R}, and $Y$ has a closed complementary space $Z\subset F$, see \cite[Lemma 4.21]{R}. The projection $P:F\to F$ along $Z$ onto $Y$ is linear and continuous (\cite[Theorem 5.16]{R}). The map $P\circ j$ is $C^1_F$-smooth and defines a $C^1_F$-map $W\to Y$. Its derivative at $b$ is an isomorphism $V\to Y$ (use $Py=y$ on $Y$ and the injectivity of $Dj(b)$). The Inverse Mapping Theorem (for maps between finite-dimensional normed spaces) yields a $C^1_F$-map
$g:Y\cap U\to V$, $U$ open in $F$ and $P(j(b))\in Y\cap U$, such that $g(P(j(b)))=b$, and an open neighbourhood $W_1\subset W$ of $b$ in $V$ such that $g(Y\cap U)=W_1$, $(P\circ j)(W_1)=Y\cap U$, $(g\circ(P\circ j))(v)=v$ on $W_1$, and $((P\circ j)\circ g)(y)=y$ on $Y\cap U$. It follows that the map $h:Y\cap U\to Z$ given by
$$
h(y)=((id_F-P)\circ j\circ g)(y)
$$ 
is $C^1_F$-smooth.

\medskip

2. Proof of $j(W_1)=\{y+h(y):y\in Y\cap U\}$ : (a) For $y\in Y\cap U$,
\begin{eqnarray*}
y+h(y) & = & y+ ((id_F-P)\circ j\circ g)(y)\\
& = & ((P\circ j)\circ g)(y)+(j\circ g)(y)-((P\circ j)\circ g)(y)=j(g(y))\in j(W_1).
\end{eqnarray*}

(b) For $x\in j(W_1)$ there exists $y\in Y\cap U$ with 
\begin{eqnarray*}
x & = & j(g(y))=((P\circ j)\circ g)(y)+j(g(y))-(P\circ j)(g(y))\\
& = & y+((id_F-P)\circ j\circ g)(y)=y+h(y).
\end{eqnarray*}
The graph representation of $j(W_1)$ now yields that it is a $C^1_F$-submanifold of $F$.
\end{proof}

\section{$C^1_{MB}$-maps which are not $C^1_F$-smooth}

Let $N$ denote the Banach space of sequences $\xi=(x_j)_1^{\infty}$ in $\R$ with limit $0$, with $|\xi|=\max_{j\in\N}|x_j|$. 
For  $j\in\N$ choose a continuously differentiable function $f_j:\R\to\R$ with $f_j(0)=0$ and $f'_j(u)=ju$ on $[-1/j,1/j]$
and
$$
\sup_{u\in\R}|f'_j(u)|\le 2\quad\text{for all}\quad j\in\N.
$$
Then the sequence $(f'_j)_1^{\infty}$ is not equicontinuous.

\medskip

For every $\xi=(x_j)_1^{\infty}\in N$ and $\eta=(y_j)_1^{\infty}\in N$ we have 
$$
|f_j(x_j)|\le 2|x_j|\quad\text{and}\quad|f_j(x_j)-f_j(y_j)|\le 2|x_j-y_j|\quad\text{for all}\quad j\in\N,
$$
and we obtain a Lipschitz continuous map 
$$
f:N\ni\xi\mapsto (f_j(x_j))_1^{\infty}\in N.
$$
Notice that for $\xi$ and $\eta$ in $N$ we also have $(f'_j(x_j)y_j)_1^{\infty}\in N$.

\medskip

\begin{prop} The map $f$ is $C^1_{MB}$-smooth, with
$$
Df(\xi)\eta=(f'_j(x_j)y_j)_1^{\infty}\quad\text{for}\quad \xi=(x_j)_1^{\infty},\eta=(y_j)_1^{\infty}.
$$
\end{prop}

\begin{proof} 1. (Directional derivatives) For $\xi=(x_j)_1^{\infty}\in N$ and $\eta=(y_j)_1^{\infty}\in N$ set $A(\xi,\eta)=(f'_j(x_j)y_j)_1^{\infty}\quad(\in N)$. For every real $h\neq0$ we have
\begin{eqnarray*}
|h^{-1}(f(\xi+h\eta)-f(\xi))-A(\xi,\eta)| & = & \sup_{j\in\N}|h^{-1}(f_j(x_j+hy_j)-f_j(x_j))-f'_j(x_j)y_j|\\
& = & \sup_{j\in\N}|\int_0^1(f'_j(x_j+\theta hy_j)y_j-f'_j(x_j)y_j)d\theta|\\
& \le & \sup_{j\in\N}\max_{0\le\theta\le1}|(f'_j(x_j+\theta hy_j)-f'_j(x_j))y_j|.
\end{eqnarray*} 
Let $\epsilon>0$. There exists $j(\epsilon)\in\N$ with
$$
|y_j|\le\frac{\epsilon}{8}\quad\text{for all integers}\quad j> j(\epsilon).
$$
For each $j\in\N$ with $j\le j(\epsilon)$ the continuity of $f'_j$ yields $h_j>0$ such that for all $h\in(-h_j,h_j)$ and for all $\theta\in[0,1]$
we have
$$
|f'_j(x_j+\theta hy_j)-f'_j(x_j)|<\frac{\epsilon}{2(|\eta|+1)}.
$$
For reals $h$ with $|h|<\min\{h_j:j\in\N,1\le j\le  j(\epsilon)\}$ we obtain
$$
|h^{-1}(f(\xi+h\eta)-f(\xi))-A(\xi,\eta)|
$$
$$
 \le\sup_{j\in\N}\max_{0\le\theta\le1}|(f'_j(x_j+\theta hy_j)-f'_j(x_j))y_j|
$$
$$
\le\max_{j=1,\ldots,j(\epsilon)}\max_{0\le\theta\le1}|(f'_j(x_j+\theta hy_j)-f'_j(x_j))y_j|
$$
$$
+ \sup_{j\in\N:j>j(\epsilon)}\max_{0\le\theta\le1}|(f'_j(x_j+\theta hy_j)-f'_j(x_j))y_j|
$$
$$
\le\frac{\epsilon|\eta|}{2(|\eta|+1)}+(2+2)\frac{\epsilon}{8}<\epsilon.
$$
We have shown that 
$$
Df(\xi)\eta=\lim_{0\neq h\to0}h^{-1}(f(\xi+h\eta)-f(\xi))
$$
exists and equals $A(\xi,\eta)$.

\medskip

2. (Continuity of $N\times N\ni(\xi,\eta)\mapsto Df(\xi)\eta\in N$) Let $\xi_0=(x_{0j})_1^{\infty}\in N$ and
$\eta_0=(y_{0j})_1^{\infty}\in N$ be given. For all $\xi=(x_j)_1^{\infty}\in N$ and $\eta=(y_j)_1^{\infty}\in N$
we have
$$
|Df(\xi)\eta-Df(\xi_0)\eta_0|\le|(Df(\xi)-Df(\xi_0))\eta|+|Df(\xi_0)(\eta-\eta_0)|,
$$
and $|Df(\xi_0)(\eta-\eta_0)|=\sup_{j\in\N}|f'_j(x_{0j})(y_j-y_{0j})|\le 2|\eta-\eta_0|$ while
$$
|(Df(\xi)-Df(\xi_0))\eta|=\sup_{j\in\N}|(f'_j(x_j)-f'_j(x_{0j}))y_j|
$$
$$
\le\sup_{j\in\N}|(f'_j(x_j)-f'_j(x_{0j}))(y_j-y_{0j})|+\sup_{j\in\N}|(f'_j(x_j)-f'_j(x_{0j}))y_{0j}|
$$
$$
\le(2+2)|\eta-\eta_0|+\sup_{j\in\N}|(f'_j(x_j)-f'_j(x_{0j}))y_{0j}|.
$$
From the preceding estimates it is obvious how to complete the proof provided we have
$$
\sup_{j\in\N}|(f'_j(x_j)-f'_j(x_{0j}))y_{0j}|\to0\quad\text{as}\quad \xi\to\xi_0.
$$
In order to prove this let $\epsilon>0$ be given. There exists $j(\epsilon)\in\N$ such that for all integers $j>j(\epsilon)$
we have $4|y_{0j}|<\frac{\epsilon}{2}$.  For each $j\in\N$ with $1\le j\le j(\epsilon)$ the continuity of $f'_j$ yields $\delta_j>0$ with
$$
|f'_j(x)-f'_j(x_{0j})|<\frac{\epsilon}{2(|\eta_0|+1)}\quad\text{for all}\quad x\in\R\quad\text{with}\quad|x-x_{0j}|<\delta_j.
$$
For every $\xi=(x_j)_1^{\infty}\in N$ with $|\xi-\xi_0|<\min_{j=1,\ldots,j(\epsilon)}\delta_j$ we get
$$
\max_{j=1,\ldots,j(\epsilon)}|(f'_j(x_j)-f'_j(x_{0j}))y_{0j}|\le\frac{\epsilon|\eta_0|}{2(|\eta_0|+1)}<\frac{\epsilon}{2}.
$$
It follows that for such $\xi$,
$$
\sup_{j\in\N}|(f'_j(x_j)-f'_j(x_{0j}))y_{0j}|\le\max_{j=1,\ldots,j(\epsilon)}|(f'_j(x_j)-f'_j(x_{0j}))y_{0j}|
$$
$$
+
\sup_{j\in\N:j>j(\epsilon)}|(f'_j(x_j)-f'_j(x_{0j}))y_{0j}|
$$
$$
<\frac{\epsilon}{2}+(2+2)\,\sup_{j\in\N:j>j(\epsilon)}|y_{0j}|\le\frac{\epsilon}{2}+\frac{\epsilon}{2}=\epsilon.
$$
\end{proof}

\begin{prop} 
There is a sequence $(\xi_k)_1^{\infty}$ in $N$ with $\lim_{k\to\infty}\xi_k=0\in N$ such that $(Df(\xi_k))_1^{\infty}$ does not converge to $Df(0)$ in the $\beta$-topology.
\end{prop}

\begin{proof} 
Recall $f'_j(1/j)=1$ for all $j\in\N$. For $k\in\N$ consider $\eta_k=(\delta_{kj})_{j=1}^{\infty}\in N$ and $\xi_k=\frac{1}{k}\eta_k\in N$. We have
$|\eta_k|=1$ for all $k\in\N$, and $\xi_k\to0$ in $N$ as $k\to\infty$. With $Df(0)=0$,
$$
|Df(\xi_k)-Df(0)|=|Df(\xi_k)|\ge|Df(\xi_k)\eta_k|=|f'_k(1/k)|=1
$$
for all $k\in\N$. 
\end{proof}

So, $f$ is not $C^1_F$-smooth.

Next, consider the Banach space   
$$
l^1=\left\{(\xi_k)_1^{\infty}\in N:\sum_1^{\infty}|\xi_k|<\infty\right\}\quad\text{with}\quad
|(\xi_k)_1^{\infty}|=\sum_1^{\infty}|\xi_k|.
$$
Obviously,
$$
\sum_1^{\infty}|f_k(\xi_k)-f_k(\eta_k)|\le 2\sum_1^{\infty}|\xi_k-\eta_k|\quad\text{for all}\quad(\xi_k)_1^{\infty}\in l^1,
\quad(\eta_k)_1^{\infty}\in l^1,
$$
and the map $f$ from Proposition 8.1 defines a map $\hat{f}:l^1\to l^1$ which is Lipschitz continuous.

\begin{prop}
$\hat{f}$ is $C^1_{MB}$-smooth with
$$
D\hat{f}(\xi)\eta=(f'_j(\xi_j)\eta_j)_1^{\infty}\quad\text{for}\quad\xi=(\xi_j)_1^{\infty}\in l^1,\quad\eta=(\eta_j)_1^{\infty}\in l^1.
$$
\end{prop}

Before giving the proof (which is similar to the proof of Proposition 8.1) consider the composition $s\circ \hat{f}:l^1\to\R$ with the continuous linear map
$$
s:l^1\ni(\xi_j)_1^{\infty}\mapsto\sum_1^{\infty}\xi_j\in\R.
$$
The composition is $C^1_{MB}$-smooth but not $C^1_F$-smooth because we have 
$D(s\circ \hat{f})(0)=sD\hat{f}(0)=0$ and, for sequences $l^1\ni\xi_k\to0\in l^1$ and $\eta_k\in l^1$ with $|\eta_k|=1$ as in the proof of Proposition 8.2,
$$
|D(s\circ \hat{f})(\xi_k)-D(s\circ \hat{f})(0)|=|sD\hat{f}(\xi_k)|\ge|sD\hat{f}(\xi_k)\eta_k|=|f'_k(1/k)|=1
$$
for all $k\in\N$, which excludes $\beta$-continuity of $D(s\circ \hat{f})$.

\medskip

{\bf Proof} of Proposition 8.3.
1. (Directional derivatives) For $\xi=(x_j)_1^{\infty}\in l^1$ and $\eta=(y_j)_1^{\infty}\in l^1$ set $A(\xi,\eta)=(f'_j(x_j)y_j)_1^{\infty}\quad(\in l^1)$. For every real $h\neq0$ we have
\begin{eqnarray*}
|h^{-1}(\hat{f}(\xi+h\eta)-\hat{f}(\xi))-A(\xi,\eta)| & = & \sum_1^{\infty}|h^{-1}(f_j(x_j+hy_j)-f_j(x_j))-f'_j(x_j)y_j|\\
& = &  \sum_1^{\infty}|\int_0^1(f'_j(x_j+\theta hy_j)y_j-f'_j(x_j)y_j)d\theta|\\
& \le &  \sum_1^{\infty}\max_{0\le\theta\le1}|(f'_j(x_j+\theta hy_j)-f'_j(x_j))y_j|.
\end{eqnarray*} 
Let $\epsilon>0$. There exists $j(\epsilon)\in\N$ with
$$
\sum_{j(\epsilon)}^{\infty}4|y_j|\le\frac{\epsilon}{2}.
$$
For each $j\in\N$ with $j\le j(\epsilon)$ we obtain from the continuity of $f'_j$ that there exists $h_j>0$ such that for all $h\in(-h_j,h_j)$ and for all $\theta\in[0,1]$
we have
$$
|f'_j(x_j+\theta hy_j)-f'_j(x_j)||y_j|2\,j(\epsilon)<\epsilon.
$$
For reals $h$ with $|h|<\min_{j=1,\ldots,j(\epsilon)}h_j$ we obtain
$$
|h^{-1}(\hat{f}(\xi+h\eta)-\hat{f}(\xi))-A(\xi,\eta)|
\le \sum_1^{\infty}\max_{0\le\theta\le1}|(f'_j(x_j+\theta hy_j)-f'_j(x_j))y_j|
$$
$$
\le\sum_1^{j(\epsilon)}\max_{0\le\theta\le1}|(f'_j(x_j+\theta hy_j)-f'_j(x_j))y_j|
+ \sum_{j(\epsilon)+1}^{\infty}\max_{0\le\theta\le1}|(f'_j(x_j+\theta hy_j)-f'_j(x_j))y_j|
$$
$$
\le\frac{\epsilon}{2}+\sum_{j(\epsilon)+1}^{\infty}(2+2)|y_j|<\epsilon.
$$
We have shown that 
$$
D\hat{f}(\xi)\eta=\lim_{0\neq h\to0}h^{-1}(\hat{f}(\xi+h\eta)-\hat{f}(\xi))
$$
exists and equals $A(\xi,\eta)$.

\medskip

2. (Continuity of $l^1\times l^1\ni(\xi,\eta)\mapsto D\hat{f}(\xi)\eta\in l^1$) Let $\xi_0=(x_{0j})_1^{\infty}\in l^1$ and
$\eta_0=(y_{0j})_1^{\infty}\in l^1$ be given. For all $\xi=(x_j)_1^{\infty}\in l^1$ and $\eta=(y_j)_1^{\infty}\in l^1$
we have
$$
|D\hat{f}(\xi)\eta-D\hat{f}(\xi_0)\eta_0|\le|(D\hat{f}(\xi)-D\hat{f}(\xi_0))\eta|+|D\hat{f}(\xi_0)(\eta-\eta_0)|,
$$
and $|D\hat{f}(\xi_0)(\eta-\eta_0)|=\sum_1^{\infty}|f'_j(x_{0j})(y_j-y_{0j})|\le 2|\eta-\eta_0|$ while
$$
|(D\hat{f}(\xi)-D\hat{f}(\xi_0))\eta|=\sum_1^{\infty}|(f'_j(x_j)-f'_j(x_{0j}))y_j|
$$
$$
\le\sum_1^{\infty}|(f'_j(x_j)-f'_j(x_{0j}))(y_j-y_{0j})|+\sum_1^{\infty}|(f'_j(x_j)-f'_j(x_{0j}))y_{0j}|
$$
$$
\le (2+2)|\eta-\eta_0|+\sum_1^{\infty}|(f'_j(x_j)-f'_j(x_{0j}))y_{0j}|.
$$
From the preceding estimates it is obvious how to complete the proof provided we have
$$
\sum_1^{\infty}|(f'_j(x_j)-f'_j(x_{0j}))y_{0j}|\to0\quad\text{as}\quad \xi\to\xi_0.
$$
In order to prove this let $\epsilon>0$ be given. There exists $j(\epsilon)\in\N$ such that for all $\xi=(x_j)_1^{\infty}\in l^1$ we have
$$
\sum_{j(\epsilon)+1}^{\infty}|(f'_j(x_j)-f'_j(x_{0j}))y_{0j}|\le\sum_{j(\epsilon)+1}^{\infty}(2+2)|y_{0j}|<
\frac{\epsilon}{2}.
$$
For each $j\in\N$ with $1\le j\le j(\epsilon)$ there exists $\delta_j>0$ with
$$
|f'_j(x_{0j}+z)-f'_j(x_{0j})||y_{0j}|2\,j(\epsilon)<\epsilon\quad\text{for all}\quad z\in(-\delta_j,\delta_j).
$$
Let $\delta=\min_{j\in\N:1\le j\le j(\epsilon)}\delta_j$.
For every $\xi=(x_j)_1^{\infty}\in l^1$ with $|\xi-\xi_0|<\delta$ we get
$$
|x_j-x_{0j}|<\delta_j\quad\text{for all}\quad j\in\{1,\ldots,j(\epsilon)\},
$$
which yields
$$
\sum_1^{j(\epsilon)}|(f'_j(x_j)-f'_j(x_{0j}))y_{0j}|\le j(\epsilon)\frac{\epsilon}{2\,j(\epsilon)}=\frac{\epsilon}{2}.
$$
It follows that
$$
\sum_1^{\infty}|(f'_j(x_j)-f'_j(x_{0j}))y_{0j}|=\sum_1^{j(\epsilon)}|(f'_j(x_j)-f'_j(x_{0j}))y_{0j}|+
\sum_{j(\epsilon)+1}^{\infty}|(f'_j(x_j)-f'_j(x_{0j}))y_{0j}|
$$
$$
<\frac{\epsilon}{2}+\frac{\epsilon}{2}=\epsilon\quad\Box.
$$

\medskip

Now we use the $C^1_{MB}$-map $f:N\to N$ from Proposition 8.1 in order to construct $C^1_{MB}$-maps $C\to N$ and $C^1\to N$ which are not $C^1_F$-smooth, for the spaces $C$ and $C^1$ with $n=1$.  Let $C_{\ast}$ denote the closed hyperplane given by $\phi(0)=0$. Then
$$
C=C_{\ast}\oplus\R\eta
$$
where $\eta(t)=1$ for all $t\le0$. Let $P_{\ast}:C\to C$ denote the projection onto $C_{\ast}$ along $\R\eta$.
Choose a strictly increasing sequence of points $t_j<0$, $j\in\N$, with limit $0$.
Choose continuous functions $e_j:(-\infty,0]\to[0,1]\subset\R$, $j\in\N$, with $e_j(t_j)=1$ for all $j\in\N$, and with mutually disjoint supports in $(-\infty,0)$. For every $x=(x_j)_1^{\infty}\in N$ the equation
$$
Jx(t)=\sum_1^{\infty}x_je_j(t)
$$
defines a continuous function $(-\infty,0]\to\R$ with $Jx(0)=0$, and 
$$
\max_{t\le0}|Jx(t)|=\max_{j\in\N}|x_j|=|x|.
$$
The map $J:N\to C_{\ast}$ is injective, linear, and continuous. 

\medskip

(1) A $C^1_{MB}$-map $C\to N$ which is not $C^1_F$-smooth. Consider the linear evaluation map
$$
E:C_{\ast}\ni\phi\mapsto(\phi(t_j))_1^{\infty}\in N.
$$
which is continuous. The map $g=f\circ E\circ P_{\ast}$ from $C$ into $N$ is $C^1_{MB}$-smooth, due to the chain rule for $C^1_{MB}$-maps. We show that it is not $C^1_F$-smooth : Otherwise the composition $g\circ J$ from $N$ into $N$ is $C^1_F$-smooth (due  to the chain rule for $C^1_F$-maps). For each $x\in N$ we have
\begin{eqnarray*}
(g\circ J)(x) & = & f(E(P_{\ast}Jx))=f(E(Jx))\quad\text{(since}\quad Jx\in C_{I{\ast}})\\
& = & f(x)
\end{eqnarray*}
which yields a contradiction to the fact that $f$ is not $C^1_F$-smooth.

\medskip

(2) A $C^1_{MB}$-map $C^1\to N$ which is not $C^1_F$-smooth. 
Let $C^1_{\ast}\subset C^1$ denote the closed hyperplane given by $\phi'(0)=0$. 
We have
$$
C^1=C^1_{\ast}\oplus\R\,\iota
$$
with $\iota(t)=t$ for all $t\le0$. Let $P^1_{\ast}:C^1\to C^1$ denote the projection onto $C^1_{\ast}$ along 
$\R\,\iota$. The equation
$$
int(\phi)(t)=-\int_t^0\phi(s)ds
$$
defines a continuous linear map $int:C\to C^1$ with $int(C_{\ast})\subset C^1_{\ast}$. Observe that $\partial(C^1_{\ast})\subset C_{\ast}$, and $\partial(int\,\phi)=\phi$ for all $\phi\in C$.
Consider $g=f\circ  E\circ\partial\circ  P^1_{\ast}$ from $C^1$ into $N$, which is $C^1_{MB}$-smooth due to the chain rule for $C^1_{MB}$-maps. We show that $g$ is not $C^1_F$-smooth : Otherwise the composition $g\circ int\circ J$ from $N$ into $N$ is $C^1_F$-smooth as well, due to the chain rule for $C^1_F$-maps. For each $x\in N$ we have
\begin{eqnarray*}
(g\circ int\circ J)(x) & = & f(E(\partial(P^1_{\ast}(int(Jx)))))\\
& = & f(E(\partial(int(Jx))))\quad\text{(since}\quad Jx\in C_{\ast}\quad\text{and}\quad int(Jx)\in C^1_{\ast})\\
& = & f(E(Jx))=f(x)
\end{eqnarray*}
which yields a contradiction to the fact that $f$ is not $C^1_F$-smooth.

\medskip

In the same way one finds examples of maps from Banach spaces $C_{ST}$ and $C^1_{ST}$ into $N$ which are
$C^1_{MB}$-smooth but not $C^1_F$-smooth.

\medskip

\begin{remark}
The examples above with their infinite-dimensional target spaces are not related to the delay differential equation (1.1).
See \cite{W10} for the construction of maps $C\to\R^n$ and $C^1\to\R^n$ which are $C^1_{MB}$-smooth but not $C^1_F$-smooth.
\end{remark}

\newpage

\begin{center}
Part II
\end{center}

\section{Examples, and the solution manifold}

We begin with the toy example (1.2),
$$
x'(t)=h(x(t-d(x(t))))
$$
with continuously differentiable functions $h:\R\to\R$ and $d:\R\to[0,\infty)\subset\R$. For continuously differentiable functions $(-\infty,t_e)$, $0<t_e\le\infty$, which satisfy Eq. (1.2) for $0\le t<t_e$ this delay differential equation has the form (1.1) for $U=C^1$ with $n=1$ and $f=f_{h,d}$ given by
$$
f_{h,d}(\phi)=h(\phi(-d(\phi(0)))).
$$
In order to see that $f_{h,d}$ is a composition of $C^1_F$-maps all defined on open sets of Fr\'echet spaces it is convenient to introduce the odd prolongation maps $P_{odd}:C\to C_{\infty}$  (with $n=1$) and $P_{odd,1}:C^1\to C^1_{\infty}$  (with $n=1$) which are defined by the relations
$$
(P_{odd}\phi)(t)=\phi(t)\quad\text{for}\quad t\le0,\quad(P_{odd}\phi)(t)=-\phi(-t)+2\phi(0)\quad\text{for}\quad t>0,
$$
and $P_{odd,1}\phi=P_{odd}\phi$ for $\phi\in C^1$. Both maps are linear and continuous. With the evaluation map
$$
ev_{\infty,1}:C^1_{\infty}\times\R\to\R,\quad ev_{\infty,1}(\phi,t)=\phi(t),
$$
we have
$$
f_{h,d}(\phi)=h\circ ev_{\infty,1}(P_{odd,1}\phi,-d(\phi(0)))
$$
for all $\phi\in C^1$. We also need the evaluation map $ev_{\infty}:C_{\infty}\times\R\to\R$ given by $ev_{\infty}(\phi,t)=\phi(t)$.

\begin{prop}
The map $ev_{\infty}$ is continuous and the map $ev_{\infty,1}$ is $C^1_F$-smooth with
$$
D\,ev_{\infty,1}(\phi,t)(\chi,s)=D_1ev_{\infty,1}(\phi,s)\chi+D_2ev_{\infty,1}(\phi,s)t=\chi(t)+s\phi'(t)
$$
\end{prop}

\begin{proof}
Arguing as in the proof of \cite[Proposition 2.1]{W7} one shows that $ev_{\infty}$ is continuous and that $ev_{\infty,1}$ is $C^1_{MB}$-smooth, and that the  directional derivatives satisfy the equations in the proposition. It remains to prove that the map 
$C^1_{\infty}\times\R\ni(\phi,t)\mapsto D\,ev_{\infty,1}(\phi,t)\in L_c(C^1_{\infty}\times\R,\R)$ is $\beta$-continuous. As  $C^1_{\infty}\times\R$ has countable neighbourhood bases it is enough to show that, given a sequence $C^1_{\infty}\times\R\ni(\phi_k,t_k)\to(\phi,t)\in C^1_{\infty}\times\R$ for $k\to\infty$, a neighbourhood $N$ of $0$ in $\R$ and a bounded subset $B\subset C^1_{\infty}\times\R$, we have
$$
(D\,ev_{\infty,1}(\phi_k,t_k)-D\,ev_{\infty,1}(\phi,t))B\subset N\quad\text{for}\quad k\quad\text{sufficiently large}.
$$ 
In order to prove this, choose $j\in\N$  with $|t|<j$ and  $|t_k|<j$ for all $k\in\N$. By \cite[Theorem 1.37]{R}, $c_j=\sup_{(\chi,s)\in B}(|\chi|_{1,\infty,j}
+|s|)<\infty$.
For every $k\in\N$ and $(\chi,s)\in B$,
$$
|(D\,ev_{\infty,1}(\phi_k,t_k)-D\,ev_{\infty,1}(\phi,t))(\chi,s)|=|\chi(t_k)-\chi(t)+s[\phi'_k(t_k)-\phi'(t)]|
$$
$$
\le\max_{-j\le u\le j}|\chi'(u)||t_k-t|+c_j(\max_{-j\le u\le j}|\phi'_k(u)-\phi'(u)|+|\phi'(t_k)-\phi'(t)|)
$$
$$
\le c_j(|t_k-t)|+|\phi_k-\phi|_{\infty,1,j}+|\phi'(t_k)-\phi'(t)|),
$$
and it becomes obvious how to complete the proof.
\end{proof}

The map $ev_1(\cdot,0):C^1\ni\phi\mapsto\phi(0)\in\R$ is linear and continuous, and the evaluation $ev:C\times(-\infty,0]\ni(\phi,t)\mapsto\phi(t)\in\R$ is continuous, see \cite[Proposition 2.1]{W7}.

\medskip

The next result says that $f_{h,d}$ satisfies the hypotheses for the results on semiflows and local invariant manifolds in the subsequent sections.

\begin{cor}
For $d:\R\to[0,\infty)\subset\R$ and $h:\R\to\R$ continuously differentiable the map $f_{h,d}$ is $C^1_F$-smooth and has property (e).
\end{cor}

\begin{proof}
The functions $d$ and $h$ are $C^1_F$-smooth. 
The map $ev_1(\cdot,0)$  is linear and continuous, hence $C^1_F$-smooth. It follows that $P_{odd,1}\times(-d\circ ev_1(\cdot,0)):C^1\to C^1_{\infty}\times\R$ is $C^1_F$-smooth, by the chain rule (Proposition 3.5) and by $C^1_F$-smoothness of maps into product spaces. Now use that
$ev_{\infty,1}$ is $C^1_F$-smooth, due to Proposition 9.1, and apply the chain rule to the composition
$$
f_{h,d}=h\circ ev_{\infty,1}\circ(P_{odd,1}\times(-d\circ\,ev_1(\cdot,0))).
$$
It follows that $f_{h,d}$ is $C^1_F$-smooth with
$$
Df_{h,d}(\phi)\chi=h'(\phi(-d(\phi(0))))[\chi(-d(\phi(0)))+\phi'(-d(\phi(0)))\{-d'(\phi(0))\}\chi(0)].
$$
For each $\phi\in C^1$ the term on the right hand side of this equation defines a linear continuation $D_ef_{h,d}(\phi):C\to\R$ of $Df_{h,d}(\phi)$. Using that the evaluation $ev$
and differentiation $C^1\to C$ are continuous one finds that
the map
$$
C^1\times C\ni(\phi,\chi)\mapsto D_ef_{h,d}(\phi)\chi\in\R
$$
is continuous.
\end{proof}

The  pantograph equation (1.3), namely,
$$
x'(t)=a\,x(\lambda t)+b\,x(t)
$$
with constants $a\in\C$, $b\in\R$ and $0<\lambda<1$, was extensively studied in \cite{KMc}. For real parameters $a,b$ and arguments $t>0$ this is a nonautonomous linear equation with unbounded delay $\tau(t)=(1-\lambda)t>0$ since $\lambda t=t-\tau(t)$. Define $F:\R\times C^1\to\R$ by
$$
F(t,\phi)=a(P_{odd,1}\phi)(-\tau(t))+b\,\phi(0),
$$
or,
$$
F=a\,ev_{\infty,1}\circ((P_{odd,1}\circ pr_2)\times(-\tau\circ pr_1)))+b\,ev_{\infty,1}(\cdot,0)\circ P_{odd,1}\circ pr_2,
$$
with the projections $pr_j$ onto the first and second component, respectively. The map $F$ is $C^1_F$-smooth, and every continuously differentiable function $x:(-\infty,t_e)\to\R$, $0<t_e\le\infty$, which satisfies the pantograph equation for $0\le t<t_e$ also solves the nonautonomous equation
\begin{equation}
x'(t)=F(t,x_t)
\end{equation}
for $0\le t<t_e$. The role of the odd prolongation map in the definition of $F$ is to allow arguments $(t,\phi)$ with $t<0$, for which $-\tau(t)>0$. The solutions of Eq. (9.1) can be obtained from the autonomous equation (1.1) with
$n=2$ and $f:C^1\to\R^2$ given by  
$$
f_1(\phi_1,\phi_2)=1,\quad f_2(\phi_1,\phi_2)=F(\phi_1(0),\phi_2)
$$
in the familiar way: If the continuously differentiable map $x:(-\infty,t_e)\to\R$ satisfies Eq. (9.1) for $t_0\le t<t_e\le\infty$ then $(s,z):(-\infty,t_e-t_0)\to\R^2$ given by
$$
s(t)=t+t_0\quad\text{and}\quad z(t)=x(t+t_0)
$$
satisfies the system
\begin{eqnarray*}
s'(t) & = & 1=f_1(s_t,z_t)\\
z'(t) & = & F(s(t),z_t)=f_2(s_t,z_t)
\end{eqnarray*}
for $0\le t<t_e$ and $s(0)=t_0$.
The map $f$ is $C^1_F$-smooth and has the extension property (e).

\medskip

For the Volterra integro-differential equation (1.4),
$$
x'(t)=\int_0^tk(t,s)h(x(s))ds
$$
with $k:\R^{n\times n}\to\R^n$ and $h:\R^n\to\R^n$ continuously differentiable the scenario is simpler than in both cases above where delays are discrete. In \cite{W9} it is shown that every continuous function $(-\infty,t_e)\to\R^n$, $0<t_e\le\infty$, which for $0<t<t_e$  is differentiable and satisfies Eq. (1.4), also satisfies an equation of the form (9.1) for $0<t<t_e$, with the $C^1_F$-map $F=F_{k,h}$ in Eq. (9.1) defined on the space $\R\times C$. The associated autonomous equation of the form (1.1) is given by the $C^1_F$-map $f_{k,h}:C((-\infty,0],\R^{n+1})\to\R^{n+1}$ with
$$
f_{k,h}=f_1\times\hat{f},\quad f_1(\psi)=1,\quad\psi=(\psi_1,\phi),\quad \hat{f}(\psi)=F(\psi_1(0),\phi).
$$
It follows that the restriction of $f_{k,h}$ to $C^1((-\infty,0],\R^{n+1})$ is $C^1_F$-smooth and has property (e), which means that the hypotheses for the theory of Eq. (1.1) in the following sections, with a semiflow on the solution manifold in $C^1((-\infty,0],\R^{n+1})$, are satisfied.  However, in the present case we also get a nice semiflow without recourse to this theory. A result in \cite{W9} for Eq. (1.1) with a map $f:C\supset U\to\R^n$ which is $C^1_F$-smooth establishes a continuous semiflow on $U$, with all solution operators $C^1_F$-smooth. In the present case, with $f=f_{k,h}$, the semiflow yields a process of solution operators for the nonautonomous equation (9.1), all of them defined on open subsets of $C((-\infty,0],\R^n)$ and $C^1_F$-smooth.  The process incorporates all solutions of the Volterra integro-differential equation. 

\medskip

The first ingredient of the present, more general theory of Eq. (1.1) is the solution manifold
$$
X_f=\{\phi\in U:\phi'(0)=f(\phi)\}.
$$

\begin{prop}
For a $C^1_F$-map $f:C^1\supset U\to\R^n$ with property (e) and $X_f\neq\emptyset$ the set
is a $C^1_F$-submanifold of codimension $n$ in the space $C^1$, with tangent spaces
$$
T_{\phi}X_f=\{\chi\in C^1:\chi'(0)=Df(\phi)\chi\}\quad\text{for all}\quad\phi\in X_f.
$$.
\end{prop}

\begin{proof}
$X_f$ is the preimage of $0\in\R^n$ under the map $C^1_F$-map $g:C^1\supset U\ni\phi\mapsto\phi'(0)-f(\phi)\in\R^n$. \cite[Proposition 2.2]{W7} applies as $f$ also is $C^1_{MB}$-smooth. It follows that all derivatives $Dg(u)$, $u\in U$, are surjective. Apply  Proposition 7.1 to $g$ and $M=\{0\}$. 
\end{proof}

\section{Evaluation maps}

For the construction of solutions of Eq. (1.1) we need a few facts about evaluation maps. 
The {\it segment evaluation maps}
$$
E_T:C_T\times(-\infty,T]\ni(\phi,t)\mapsto \phi_t\in C,
$$
$$
E^1_T:C^1_T\times(-\infty,T]\ni(\phi,t)\mapsto \phi_t\in C^1
$$
and 
$$
E^{10}_T:C^1_T\times(-\infty,T]\ni(\phi,t)\mapsto \phi_t\in C
$$
for $T\in\R$ and  their analogues $E_{\infty},E^1_{\infty}$ for $T=\infty$ are all 
linear in the first argument.

\begin{prop}
Let $T\le\infty$.

(i) The maps $E_T$ and $E^1_T$ are continuous.

(ii) For every $\phi\in C^1_T$ the curve $\Phi:(-\infty,T)\ni t\mapsto\phi_t\in C$ is continuously differentiable, with $\Phi'(t)=E_T(\partial_T\phi,t)$.

(iii) The map $E^{10}_T|_{C^1_T\times(-\infty,T)}$ is $C^1_F$-smooth, with
\begin{eqnarray*}
D_1E^{10}_T(\phi,t)\hat{\phi} & = & E^{10}_T(\hat{\phi},t)=\hat{\phi}_t\quad\text{and}\\
D_2E^{10}_T(\phi,t)s & = & s\,E_T(\partial_T\phi,t)=s(\partial_T\phi)_t=s(\phi')_t.
\end{eqnarray*}
\end{prop}

\begin{proof}
1. For assertions (i) and (ii), and for the fact that the map $E^{10}_T|_{C^1_T\times(-\infty,T)}$ is $C^1_{MB}$, and for the formulae for the partial derivatives in assertion (iii) see the proof of \cite[ Proposition 3.1]{W7}. It remains to show that
$DE^{10}_T$ is $\beta$-continuous. Let a sequence $(\phi_j,t_j)_1^{\infty}$ in $C^1_T\times(-\infty,T)$ be given which converges to some $(\phi,t)\in C^1_T\times(-\infty,T)$.  Let a bounded subset $B\subset C^1_T\times\R$ and a neighbourhood
$V$ of $0$ in $C$ be given. We may assume
$$
V=\{\chi\in C:|\chi|_{l}<\frac{1}{l}\}\quad\text{for some integer}\quad l>0
$$
and have to show that for $j$ sufficiently large,
$$
(DE^{10}_T(\phi_j,t_j)-DE^{10}_T(\phi,t))B\subset V.
$$
For every $(\hat{\phi},\hat{t})\in C^1_T\times\R$ and for all $j\in\N$,
$$
(DE^{10}_T(\phi_j,t_j)-DE^{10}_T(\phi,t))(\hat{\phi},\hat{t})=\hat{\phi}_{t_j}-\hat{\phi}_t+\hat{t}[(\phi_j')_{t_j}-(\phi')_t].
$$
2. As the projections from $C^1_T\times\R$ onto $C^1_T$ and onto $\R$ are continuous and linear they map the bounded set $B$
into bounded sets, and we obtain that for some real $r>0$ and for all $k\in\N$,
$$
\{\hat{t}\in\R:\text{For some}\quad\hat{\phi}\in C^1_T,(\hat{\phi},\hat{t})\in B\}\subset[-r,r]
$$
and
$$
\sigma_{1,T,k}=\sup\{|\hat{\phi}|_{1,T,k}\in\R:\text{For some}\quad\hat{t}\in\R,(\hat{\phi},\hat{t})\in B\}<\infty.
$$
3. Choose $k\in\N$ so large that for all $s\in[-l,0]$ and for all $j\in\N$,
$$
T-k<s+t_j<T\quad\text{and}\quad T-k<s+t<T.
$$
Consider $(\hat{\phi},\hat{t})\in B$. For each $j\in\N$ we have
$$
|\hat{\phi}_{t_j}-\hat{\phi}_t|_l=\max_{-l\le s\le0}|\hat{\phi}(t_j+s)-\hat{\phi}(t+s)|
$$
$$
\le\max_{T-k\le u\le T}|\hat{\phi}'(u)||t_j-t|\le\sigma_{1,T,k}|t_j-t|
$$
and
$$
|\hat{t}[(\phi_j')_{t_j}-(\phi')_t]|_l\le r[|(\phi_j')_{t_j}-(\phi')_{t_j}|_l+|(\phi')_{t_j}-(\phi')_t|_l]
$$
$$
\le r[\max_{-l\le s\le 0}|\phi_j'(s+t_j)-\phi'(s+t_j)|+\max_{-l\le s\le 0}|(\phi')(s+t_j)-(\phi')(s+t)|]
$$
$$
\le r[\max_{T-k\le u\le T}|\phi_j'(u)-\phi'(u)|+\max_{-l\le s\le 0}|(\phi')(s+t_j)-(\phi')(s+t)|].
$$
Altogether, for every $j\in\N$ and for all $(\hat{\phi},\hat{t})\in B$,
$$
|(DE^{10}_T(\phi_j,t_j)-DE^{10}_T(\phi,t))(\hat{\phi},\hat{t})|_l
$$
$$
\le\sigma_{1,T,k}|t_j-t|
+r[|\phi_j-\phi|_{1,T,k}+\max_{-l\le s\le 0}|(\phi')(s+t_j)-(\phi')(s+t)|].
$$
Using $t_j\to t$ and $|\phi_j-\phi|_{1,T,k}\to0$ as $j\to\infty$ and the uniform continuity of $\phi'$ on $[T-k,T]$ one finds
$J\in\N$ such that for all integers $j\ge J$ and for all  $(\hat{\phi},\hat{t})\in B$,
$$
|(DE^{10}_T(\phi_j,t_j)-DE^{10}_T(\phi,t))(\hat{\phi},\hat{t})|_l<\frac{1}{l}
$$
It follows that $(DE^{10}_T(\phi_j,t_j)-DE^{10}_T(\phi,t))B\subset V$ for all integers $j\ge J$.
\end{proof}

\section{The fixed point problem, and a substitution operator}

In the sequel we always assume that $U\subset C^1$ is open and that $f:U\to\R^n$ is $C^1_F$-smooth and has the property (e). 

\medskip

Following \cite{W7} we rewrite the initial value problem 
\begin{equation}
x'(t)=f(x_t)\quad\text{for}\quad t\ge0,\quad x_0=\phi\in X_f
\end{equation}
as a fixed point equation: Suppose $x:(-\infty,T]\to\R^n$, $T>0$, is a solution of Eq. (1.1) on $[0,T]$ with $x_0=\phi$. Extend $\phi$ by $\phi(t)=\phi(0)+t\phi'(0)$ to a continuously differentiable function $\hat{\phi}:(-\infty,T]\to\R^n$. Then $y=x-\hat{\phi}$ satisfies $y(t)=0$ for $t\le0$, the curve $(-\infty,T]\ni s\mapsto x_s\in C^1$ is continuous (use $x_s=E^1_T(x,s)$ and apply Proposition 10.1 (i)), as well as the curves  $(-\infty,T]\ni s\mapsto y_s\in C^1$ and $(-\infty,T]\ni s\mapsto \hat{\phi}_s\in C^1$. For $0\le t\le T$ we get
\begin{eqnarray*}
y(t) & = & x(t)-\hat{\phi}(t)=x(0)+\int_0^tf(x_s)ds-\phi(0)-t\phi'(0)\\
& = & \int_0^tf(y_s+\hat{\phi}_s)ds-t\,f(\phi)\\
& = & \int_0^tf(y_s+\hat{\phi}_s)-f(\phi))ds.
\end{eqnarray*}holds
Obviously, $y(0)=0=y'(0)$. So $\eta=y|_ {[0,T]}\in C^1_{0T,0}$ satisfies the fixed point equation
\begin{equation}
\eta(t)=\int_0^t(f(\hat{\eta}_s+\hat{\phi}_s)-f(\phi))ds,\quad 0\le t\le T,
\end{equation}
where $\hat{\eta}\in C^1_T$ is the prolongation of $\eta$ given by $\hat{\eta}(t)=0$ for all $t<0$.
In order to find a solution of the initial value problem (11.1) one solves the fixed point equation (11.2) by means of a parametrized contraction on a subset of the Banach space $C^1_{0T,0}$ with the parameter $\phi\in U$ in the Fr\'echet space $C^1$. For $\phi\in X_f$ the associated fixed point $\eta=\eta_{\phi}$ yields a solution $x=\hat{\eta}+\hat{\phi}$ of the initial value probem (11.1).

\medskip

The application of a suitable contraction mapping theorem, namely, Theorem 5.2, requires some preparation. We begin with the substitution operator
$$
F_T:dom_T\to C_{0T}
$$
which for $0<T<\infty$ is given by
$$
dom_T=\{\phi\in C^1_T:\text{For}\,\,0\le s\le T,\phi_s\in U\}
$$
and
$$
F_T(\phi)(t)=f(\phi_t)=f(E^1_T(\phi,t))\in\R^n.
$$
\cite[Proposition 3.2]{W7} guarantees that for $0<T<\infty$ the domain $dom_T$ is open and that $F_T$ is a $C^1_{MB}$-map with
$$
(DF_T(\phi)\hat{\phi})(s)=D_ef(E^1_T(\phi,s))E^{10}_T(\hat{\phi},s).
$$
(Notice that  in order to obtain that $F_T$ is $C^1_{MB}$ the chain rule can not be applied, due to lack of smoothness of the map $E^1_T$.)

\begin{prop}
The map $F_T$, $0<T<\infty$,  is $C^1_F$-smooth.
\end{prop}

\begin{proof}
1. Let $\phi\in dom_T\subset C^1_T$, $\epsilon>0$, and a bounded set $B\subset C^1_T$ be given. Using the norm on $C_{0T}$ we have to find a neighbourhood $N$ of $\phi$ in $C^1_T$ so that for every $\psi\in N$ and for all $\hat{\phi}\in B$,
$$
\max_{0\le s\le T}|((DF_T(\psi)-DF_T(\phi))\hat{\phi})(s)|<\epsilon.
$$
Define
$$
B_T=\{E^1_T(\hat{\phi},s)\in C^1:0\le s\le T,\hat{\phi}\in B\}
$$
Claim: $B_T\subset C^1$ is bounded.

\medskip

Proof: Consider a seminorm $|\cdot|_{1,j}$, $j\in\N$. Choose an integer $k\ge j+T$. The seminorm $|\cdot|_{1,T,k}$ is bounded on $B$. For every $\hat{\phi}\in B$ and for all $s\in[0,T]$ we see from
\begin{eqnarray*}
|E^1_T(\hat{\phi},s)|_{1,j} & = & \max_{-j\le u\le0}|\hat{\phi}(s+u)|+\max_{-j\le u\le0}|\hat{\phi}'(s+u)|\\
& \le &  \max_{-j\le w\le T}|\hat{\phi}(w)|+\max_{-j\le w\le T}|\hat{\phi}'(w)|\le|\hat{\phi}|_{1,T,k}
\end{eqnarray*}
that $|\cdot|_{1,j}$ is bounded on $B_T$.

\medskip

2. For every $\psi\in dom_T$, $\hat{\phi}\in B$, $s\in[0,T]$ we have
\begin{eqnarray*}
((DF_T(\psi)-DF_T(\phi))\hat{\phi})(s) & = & (D_ef(E^1_T(\psi,s))-D_ef(E^1_T(\phi,s))E^{10}_T(\hat{\phi},s)\\
& & \text{(see \cite[Proposition 3.2]{W7})}\\
& = & (Df(E^1_T(\psi,s))-Df(E^1_T(\phi,s))E^1_T(\hat{\phi},s)\\
& & \text{(with}\quad E^{10}_T(\hat{\phi},s)\in C^1),
\end{eqnarray*}
where $E^1_T(\hat{\phi},s)$ is in the bounded set $B_T$. As $f$ is $C^1_F$-smooth the composition
$$
C^1_T\times\R\supset dom_T\times[0,T]\ni(\psi,s)\mapsto Df(E^1_T(\psi,s))\in L_c(C^1,\R^n)
$$
is $\beta$-continuous, hence uniformly $\beta$-continuous on the compact set $\{\phi\}\times[0,T]$ (see Proposition 1.2). It follows that there is a neighbourhood $N$ of $\phi$ in $C^1_T$ such that for every $\psi\in N$ and for all $s\in[0,T]$ the difference
$$
Df(E^1_T(\psi,s))-Df(E^1_T(\phi,s)
$$
is contained in the neighbourhood $U_{U_{\epsilon}(0),B_T}$ of $0$ in $L_c(C^1,\R^n)$, with
$U_{\epsilon}(0)=\{x\in\R^n:|x|<\epsilon\}$. Finally, we obtain for each $\psi\in N$, $s\in[0,T]$, $\hat{\phi}\in B$,
\begin{eqnarray*}
|((DF_T(\psi)-DF_T(\phi))\hat{\phi})(s)| & = & |(Df(E^1_T(\psi,s))-Df(E^1_T(\phi,s))E^1_T(\hat{\phi},s)|\\
& < & \epsilon.
\end{eqnarray*}
\end{proof}

The {\it prolongation maps}
$$
P_T:C^1\to C^1_T,\quad 0<T\le\infty,
$$
given by
$$
P_T\phi(t)=\phi(t)\quad\text{for}\quad t\le0,\quad P_T\phi(t)=\phi(0)+t\phi'(0)\quad\text{for}\quad 0<t\le T,
$$
$$
P_{ST}:C^1_{0S}\to C^1_{0T},\quad 0<S<T<\infty,
$$
given by
$$
P_{ST}\phi(t)=\phi(t)\quad\text{for}\quad 0\le t\le S,\quad P_{ST}\phi(t)=\phi(S)+(t-S)\phi'(S)\quad\text{for}\quad S<t\le T,
$$
$$
Z_T:C_{0T,0}\to C_T,\quad 0<T<\infty
$$
given by
$$
Z_T\phi(t)=\phi(t)\quad\text{for}\quad 0\le t\le T,\quad Z_T(\phi)(t)=0\quad\text{for}\quad t<0,
$$
and the integration operators
$$
I_T:C_{0T,0}\to C^1_{0T,0},\quad 0<T<\infty,\quad\text{given by}\quad I_T\phi(t)=\int_0^t\phi(s)ds
$$
are all linear and continuous. We have $Z_TC^1_{0T,0}\subset C^1_T$, and the induced map $C^1_{0T,0}\stackrel{Z_T}{\to} C^1_T$ is continuous, too. For $P_{ST}$, $0<S<T$, 
$$
P_{ST}C^1_{0S,0}\subset C^1_{0T,0},
$$
and
\begin{equation*}
|P_{ST}\phi|_{1,0T}\le(2+T)|\phi|_{1,0S}\quad\text{for all}\quad\phi\in C^1_{0S}
\end{equation*}
because of the estimate
\begin{eqnarray*}
|P_{ST}\phi|_{1,0T} & = & \max_{0\le t\le T}|P_{ST}\phi(t)|+\max_{0\le t\le T}|(P_{ST}\phi)'(t)|\\
& \le & \max_{0\le t\le S}|\phi(t)|+|\phi(S)|+|\phi'(S)|T+ \max_{0\le t\le S}|\phi'(t)|.
\end{eqnarray*}

It follows that for every $T>0$ the set
$$
D_T=\{(\phi,\eta)\in U\times C^1_{0T,0}: P_T\phi+Z_T\eta\in dom_T\}
$$
is open. Let $pr_1$ and $pr_2$ denote the projections from $C^1\times C^1_{0T,0}$ onto the first and second factor, respectively. Define $\tau:\R^n\mapsto C_{0T}$ by $\tau(\xi)(t)=\xi$. Both projections and $\tau$ are continuous linear maps. Using Proposition 11.1, the chain rule, and linearity of differentiation we  infer that
the map
$$
G_T:C^1\times C^1_{0T,0}\supset U\times C^1_{0T,0}\supset D_T\to C_{0T,0}\subset C_{0T}
$$
given by
$$
G_T(\phi,\eta)=F_T( P_Tpr_1(\phi,\eta)+Z_Tpr_2(\phi,\eta))-\tau\circ f\circ pr_1(\phi,\eta)
$$
$$
\text{(notice that}\quad G_T(\phi,\eta)(0)=f((P_T\phi+Z_T\eta)_0)-f(\phi)=f(\phi+0)-f(\phi)=0)
$$
is $C^1_F$-smooth. For the derivatives we obtain the following result.

\begin{cor}
Let $0<T<\infty$. For $(\phi,\eta)\in D_T$  and $\hat{\phi}\in C^1,\hat{\eta}\in C^1_{0T,0}$,
$$
DG_T(\phi,\eta)(\hat{\phi},\hat{\eta})=DF_T(P_T\phi+Z_T\eta)(P_T\hat{\phi}+Z_T\hat{\eta})-\tau(Df(\phi)\hat{\phi}),
$$
and for $0\le t\le T$,
\begin{eqnarray*}
DG_T(\phi,\eta)(\hat{\phi},\hat{\eta})(t) & = & (D_ef(E^1_T(P_T\phi+Z_T\eta,t)E^{10}_T(P_T\hat{\phi}+Z_T\hat{\eta},t)\\
& & -\tau(Df(\phi)\hat{\phi})(t)\\
& = & D_ef((P_T\phi)_t+(Z_T\eta)_t)((P_T\hat{\phi})_t+(Z_T\hat{\eta})_t)\\
& & -Df(\phi)\hat{\phi}.
\end{eqnarray*}
\end{cor}

The map $A_T=I_T\circ G_T$ is $C^1_F$-smooth. We now restate \cite[Proposition 3.4]{W7}, which prepares the proof that $A_T$ with $T>0$ sufficiently small defines a uniform contraction on a small ball in $C^1_{0T,0}$.

\begin{prop}
Let $\phi\in U$ be given. There exist $T=T_{\phi}>0$, a neighbourhood $V=V_{\phi}$ of $\phi$ in $U$, $\epsilon=\epsilon_{\phi}>0$, and $j=j_{\phi}\in\N$ such that for all $S\in(0,T)$, $\chi\in V$, $\eta$ and $\tilde{\eta}$ in $C^1_{0S,0}$ with $|\eta|_{1,0S}<\epsilon$ and $|\tilde{\eta}|_{1,0S}<\epsilon$, $w\in[0,S]$, and $\theta\in[0,1]$, 
\begin{equation*}
(P_S\chi)_w+(Z_S\eta)_w+\theta[(Z_S\tilde{\eta})_w-(Z_S\eta)_w]\in U
\end{equation*}
and
$$
|D_ef((P_S\chi)_w+(Z_S\eta)_w+\theta[(Z_S\tilde{\eta})_w-(Z_S\eta)_w])[(Z_S\tilde{\eta})_w-(Z_S\eta)_w)]|$$
$$
\le 2j\,|\tilde{\eta}-\eta|_{0S}.
$$
\end{prop}

\begin{proof}
See the proof of \cite[Proposition 3.4]{W7}
\end{proof}

Let $\phi\in U$, and let $T=T_{\phi}>0$, a convex neighbourhood $V=V_{\phi}$ of $\phi$ in $U$, $\epsilon=\epsilon_{\phi}>0$, and $j=j_{\phi}\in\N$ be given as in Proposition  11.3.

Then Propositions 4.1, 4.2, 4.3 from \cite{W7} hold, with verbatim the same proofs. We restate these propositions as follows.

\begin{prop}
For every $S\in(0,T)$, $\chi\in V$, $\eta$ and $\tilde{\eta}$ in $C^1_{0S,0}$ with $|\eta|_{1,0S}<\epsilon$ and
$|\tilde{\eta}|_{1,0S}<\epsilon$,
$$(\chi,\eta)\in D_S,\,(\chi,\tilde{\eta})\in D_S,\quad\text{and}\quad
|A_S(\chi,\tilde{\eta})-A_S(\chi,\eta)|_{1,0S}\le 2jS(S+1)|\tilde{\eta}-\eta|_{1,0S}.
$$
\end{prop}

\begin{prop}
$\lim_{S\searrow0}A_S(\phi,0)=0$.
\end{prop}

\begin{prop}
There exist $S_{\phi}\in(0,T_{\phi})$ and an open neighbourhood $W_{\phi}$ of $\phi$ in $V_{\phi}$ such that for all $\chi\in W_{\phi}$, for all $S\in(0,S_{\phi}]$, and all
$\eta\in C^1_{0S,0}$ and $\tilde{\eta}\in C^1_{0S,0}$ with $|\eta|_{1,0S}\le\frac{\epsilon_{\phi}}{2}$ and  $|\tilde{\eta}|_{1,0S}\le\frac{\epsilon_{\phi}}{2}$,
$$
(\chi,\eta)\in D_S,\,(\chi,\tilde{\eta})\in D_S,
$$
$$
|A_S(\chi,\eta)|_{1,0S}<\frac{\epsilon_{\phi}}{2}\quad\text{and}\quad
|A_S(\chi,\tilde{\eta})-A_S(\chi,\eta)|_{1,0S}\le\frac{1}{2}|\tilde{\eta}-\eta|_{1,0S}.
$$
\end{prop}

For each $S\in(0,S_{\phi}]$ now the uniform contraction result Theorem 5.2 applies to the map 
$$
W_{\phi}\times\{\eta\in C^1_{0S,0}:|\eta|_{1,0S}<\epsilon_{\phi}\}\ni(\chi,\eta)\mapsto A_S(\chi,\eta)\in C^1_{0S,0},
$$
with $M=M_{\phi}=\{\eta\in C^1_{0S,0}:|\eta|_{1,0S}\le\frac{\epsilon_{\phi}}{2}\}$, and yields a $C^1_F$-map
$$
W_{\phi}\ni\chi\mapsto\eta_{\chi}\in C^1_{0S,0}
$$
given by $\eta_{\chi}\in M_{\phi}$ and $A_S(\chi,\eta_{\chi})=\eta_{\chi}$. As the maps $P_S$ and $C^1_{0S,0}\stackrel{Z_S}{\to}C^1_S$ are linear and continuous it follows from linearity of differentiation and by means of the chain rule that also the map
$$
\Sigma_{\phi}:W_{\phi}\ni\chi\mapsto P_S\chi+Z_S\eta_{\chi}\in C^1_S
$$
is $C^1_F$-smooth. An application of the chain rule to the compositions of this map with the continuous linear maps $E^1_S(\cdot,t):C^1_S\to C^1$, $0\le t\le S$, yields that all maps
$$
W_{\phi}\ni\chi\mapsto E^1_S(\Sigma_{\phi}(\chi),t)\in C^1,\quad 0\le t\le S,
$$
are $C^1_F$-smooth.  As $E^1_S$ is continuous we obtain that the composition
$$
[0,S]\times W_{\phi}\ni(t,\chi)\mapsto E^1_S(\Sigma_{\phi}(\chi),t)\in C^1
$$ 
is continuous.

\cite[Proposition 4.4]{W7} showed that the restriction of the map $\Sigma_{\phi}$ to the solution manifold provides us with solutions of the initial value problem (11.2). It remains valid, with the same proof, and is restated as follows.

\begin{prop}
For every $S\in(0,S_{\phi}]$ and for every $\chi\in W_{\phi}\cap X_f$ the function $x=x^{(\chi)}=\Sigma_{\phi}(\chi)$ is  a solution of Eq. (1.1) on $[0,S]$, with $x_0=\chi$ and $x_t\in X_f$ for $0 \le t\le S$.
\end{prop}

The restriction
$$
[0,S]\times (W_{\phi}\cap X_f)\ni(t,\chi)\mapsto E^1_S(\Sigma_{\phi}(\chi),t)\in C^1
$$
is continuous, and the restrictions 
$$
W_{\phi}\cap X_f\ni\chi\mapsto E^1_S(\Sigma_{\phi}(\chi),t)\in C^1,\quad 0\le t\le S,
$$
are $C^1_F$-maps from their domains in the $C^1_F$-submanifold $X_f$ into $C^1$. 

\section{The semiflow on the solution manifold}

The uniqueness results \cite[Propositions 4.5 and 5.1]{W7} remain valid, with the same proofs. As in \cite[Section 5]{W7}
we find maximal solutions $x^{\phi}:(-\infty,t_{\phi})\to\R^n$, $0<t_{\phi}\le\infty$, of the initial value problems
$$
x'(t)=f(x_t)\quad\text{for}\quad t>0,\quad x_0=\phi\in X_f,
$$
which are solutions on $[0,t_{\phi})$ and have the property that any other solution on some interval with left endpoint $0$, of the same initial value problem, is a restriction of $x^{\phi}$. The relations
$$
\Omega_f=\{(t,\phi)\in[0,\infty)\times X_f:t<t_{\phi}\},\quad\Sigma_f(t,\phi)=x^{\phi}_t
$$
define a semiflow $\Sigma_f:\Omega_f\to X_f$ on $X_f$, compare \cite[Proposition 5.2]{W7}. In \cite[Proposition 5.3]{W7} and in its proof the words {\it continuously differentiable} can everywhere be replaced by the expression $C^1_F$-smooth. Thus 
$\Sigma_f$ is continuous, with each domain
$$
\Omega_{f,t}=\{\phi\in X_f:(t,\phi)\in\Omega_f\},  \quad t\ge0,
$$
an open subset of $X_f$ and the
time-$t$-map
$$
\Sigma_{f,t}:\Omega_{f,t}\to X_f,\quad\Sigma_{f,t}(\phi)=\Sigma_f(t,\phi),\quad t\ge0,
$$
$C^1_F$-smooth in case $\Omega_{f,t}\neq\emptyset$.

\cite[Proposition 5.5]{W7} and its proof remain valid. In the proof of \cite[Proposition 6.1]{W7} the words {\it continuously differentiable} can everywhere be replaced by the expression $C^1_F$-smooth. This yields
$$
D_2\Sigma_f(t,\phi)\chi=D\Sigma_{f,t}(\phi)\chi=v_t^{\phi,\chi}
$$
with the unique maximal continuously differentiable solution $v=v^{\phi,\chi}$ of the initial value problem
$$
v'(t)=Df(\Sigma_f(t,\phi))v_t\quad\text{for}\quad t>0,\quad v_0=\chi\in T_{\phi}X_f.
$$

\newpage

\begin{center}
Part III
\end{center}

\section{On locally bounded delay, the extension property, and prolongation and restriction}

Assume as in Part II that $f:C^1\supset U\to\R^n$ is $C^1_F$-smooth and that $f$ has property (e). It is convenient from here on to abbreviate $X=X_f$, $\Omega=\Omega_f$, and $\Sigma=\Sigma_f$. Let a stationary point $\bar{\phi}\in X$ of $\Sigma$ be given, $\Sigma(t,\bar{\phi})=\bar{\phi}$ for all $t\ge0$. Then $\bar{\phi}$ is constant. (Proof of this: The solution $x$ of Eq. (1.1) on $[0,\infty)$ with $x_0=\phi$ satisfies 
$x(t)=x_t(0)=\Sigma(t,\phi)(0)=\phi(0)$ for all $t\ge0$. For all $s<0$ we have $x(s)=\phi(s)=\Sigma(-s,\phi)(s)=
x_{-s}(s)=x(0)=\phi(0)$.)

\medskip

Choose an open neighbourhood $N$ of $\bar{\phi}$ in $U$ and $d>0$ according to property (lbd). We restate
\cite[Proposition 2.1]{W8} as follows.

\begin{prop}
For every $\phi\in N$ we have
\begin{equation*}
Df(\phi)\psi=0\quad\text{for all}\quad\psi\in C^1\quad\text{with}\quad\psi(s)=0\quad\text{on}\quad[-d,0],
\end{equation*}
and
\begin{equation*}
D_ef(\phi)\chi=0\quad\text{for all}\quad\chi\in C\quad\text{with}\quad\chi(s)=0\quad\text{on}\quad[-d,0].
\end{equation*}
\end{prop}

Set $\bar{\phi}_d=R_{d,1}\bar{\phi}=\bar{\phi}|_{[-d,0]}$. As $\bar{\phi}$ is constant we have 
\begin{equation*}
P_{d,1}\bar{\phi}_d=\bar{\phi}\in N,
\end{equation*} 
and it follows that there exist neighbourhoods $U_d$ of $\bar{\phi}_d$ in $C^1_d$ with $P_{d,1}U_d\subset N$. Due to the chain rule the map
$$
f_d:C^1_d\supset U_d\to\R^n,\quad f_d(\phi)=f(P_{d,1}\phi),
$$ 
is $C^1_F$-smooth, with
\begin{equation*}
Df_d(\phi)\chi=Df(P_{d,1}\phi)P_{d,1}\chi.
\end{equation*}

According to \cite[Proposition 2.2]{W8} $f_d$ has property (e). Results from \cite{W1,W2} apply and show that the equation
\begin{equation}
x'(t)=f_d(x_t)
\end{equation}
(with segments $x_t:[-d,0]\ni s\mapsto x(t+s)\in\R^n$) defines a continuous semiflow $\Sigma_d:\Omega_d\to X_d$ on the submanifold 
\begin{equation*}
X_d=\{\phi\in U_d:\phi'(0)=f_d(\phi)\},\quad\text{codim}\,X_d=n,
\end{equation*}
of the Banach space $C^1_d$. In the terminology of the present paper, the manifold $X_d$ and all solution operators $\Sigma_d(t,\cdot)$, $t\ge0$, with non-empty domain are $C^1_F$-smooth.

\medskip

The proofs of \cite[Propositions 2.3-2.5]{W8} remain valid without change. We restate the result as follows.

\begin{prop}
(i) $X_d=R_{d,1}(X\cap N\cap R_{d,1}^{-1}(U_d))$

\medskip

(ii) For every $\phi\in X\cap N\cap R_{d,1}^{-1}(U_d)$, 
\begin{equation*}
T_{R_{d,1}\phi}X_d=R_{d,1}T_{\phi}X.
\end{equation*}
(iii) For $(t,\phi)\in\Omega$ with $\Sigma([0,t]\times\{\phi\})\subset N\cap R_{d,1}^{-1}(U_d)$,
\begin{equation*}
(t,R_{d,1}\phi)\in \Omega_d\quad\text{and}\quad \Sigma_d(t,R_{d,1}\phi)=R_{d,1}\Sigma(t,\phi).
\end{equation*}
(iv) If $(t,\chi)\in\Omega_d$ and if $x:(-\infty,t]\to\R^n$ given by $x(s)=x^{\chi}(s)$ on $[-d,t]$ and by $x(s)=(P_{d,1}\chi)(s)$ for $s<-d$ satisfies $\{x_s:0\le s\le t\}\subset N$ then
\begin{equation*}
(t,P_{d,1}\chi)\in\Omega\quad\text{and}\quad R_{d,1}\Sigma(t,P_{d,1}\chi)=\Sigma_d(t,\chi).
\end{equation*}
\end{prop}

Proposition 13.2 (iii) shows that $\bar{\phi}_d$ is a stationary point of the semiflow $\Sigma_d$.  

\medskip

For $t\ge0$ consider the operators $T_t=D_2\Sigma(t,\bar{\phi})$ on $T_{\bar{\phi}}X$ and $T_{d,t}=D_2\Sigma_d(t,\bar{\phi}_d)$ on $T_{\bar{\phi}_d}X_d$. The proof of \cite[Corollary 2.6]{W8} remains valid. 
We state the result as follows.

\begin{cor}
(i) For $(t,\phi)\in\Omega$ as in Proposition 13.2 (iii) and for all $\chi\in T_{\phi}X$,
\begin{equation*}
R_{d,1}\chi\in T_{R_{d,1}\phi}X_d\quad\text{and}\quad R_{d,1}D_2\Sigma(t,\phi)\chi=D_2\Sigma_d(t,R_{d,1}\phi)R_{d,1}\chi.
\end{equation*}
(ii) For all $\chi\in T_{\bar{\phi}}X$ and for all $t\ge0$,
\begin{equation*}
R_{d,1}\chi\in T_{\bar{\phi}_d}X_d\quad\text{and}\quad R_{d,1}T_t\chi=T_{d,t}R_{d,1}\chi.
\end{equation*}
\end{cor}

From \cite[Sections 3.5 and 4.1-4.3]{HKWW} and from \cite{K2} we get
local stable, center, and unstable manifolds of $\Sigma_d$  at $\bar{\phi}_d\in X_d\subset C^1_d$, all of them 
$C^1_F$-smooth. 

\section{Decomposition of the tangent space}

Let $Y=T_{\bar{\phi}}X$. In this section we recall from \cite[Section 3]{W8} the definitions of the linear stable, center, and unstable spaces of the operators $T_t:Y\to Y$, $t\ge0$.  

\medskip

The linear stable space in $Y$ is defined by
$$
Y_s=R_{d,1}^{-1}Y_{d,s}
$$ 
with the linear stable space $Y_{d,s}$ of the strongly continuous semigroup
$(T_{d,t})_{t\ge0}$ on the tangent space $Y_d=T_{\bar{\phi}_d}X_d\subset C^1_d$. 
We have $Y_{d,s}=Y_d\cap C_{d,s}$  with the linear stable space $C_{d,s}$ of the strongly continuous semigroup of solution operators $T_{d,e,t}:C_d\to C_d$, $t\ge0$, which is defined by the equation
\begin{equation}
v'(t)=D_ef_d(\bar{\phi}_d)v_t.
\end{equation}
Let $C_{d,c}$ and $C_{d,u}$ denote the finite-dimensional linear center and unstable spaces of the semigroup on $C_d$. Each $\chi\in C_{d,c}\oplus C_{d,u}$ uniquely defines an analytic solution $v=v^{\chi}$
on $\R$ of Eq. (14.1). The injective map
$$
I:C_{d,c}\oplus C_{d,u}\ni\chi\mapsto\chi|_{(-\infty,0]}\in C^1
$$
is linear, and continuous (as its domain is finite-dimensional) . The center and unstable spaces in $Y$ are defined as
$$
Y_c=IC_{d,c}\quad\text{and}\quad Y_u=IC_{d,u},
$$
respectively. They are finite-dimensional and
the maps $T_t$, $t\ge0$, act as isomorphisms on each of them. The stable space $Y_s$ is closed and positively invariant under each map $T_t$, $t\ge0$, and we have the decomposition
$$
Y=Y_s\oplus Y_c\oplus Y_u.
$$
Finally, observe
$$
Y_u\subset B^1_a
$$
since each $v^{\chi}$, $\chi\in C_{d,u}$, and its derivative both have limit $0$ at $-\infty$.

\section{The local stable manifold}

We begin with the local stable manifold  $W^s_d\subset X_d$ of the semiflow $\Sigma_d$ at the stationary point $\bar{\phi}_d\in X_d\subset C^1_d$ as it was obtained in \cite{HKWW}.  It is easy to see that 
$W^s_d$ is a continuously differentiable submanifold of the Banach space $C^1_d$ which is locally positively invariant under $S_d$, with tangent space
\begin{equation*}
T_{\bar{\phi}_d}W^s_d=Y_{d,s}
\end{equation*}
at $\bar{\phi}_d$, and that it has the following poperties (I) and (II), for some 
$\beta>0$ chosen with
\begin{equation*}
\Re\,z<-\beta<0
\end{equation*}
for all $z$ with $\Re\,z<0$  in the spectrum of the generator of the semigroup on $C_d$, and for some $\gamma>\beta$.

\medskip

(I) There are an open neighbourhood $\tilde{W}^s_d$ of $\bar{\phi}_d$ in $W^s_d$ such that
$[0,\infty)\times\tilde{W}^s_d\subset\Omega_d$ and $\Sigma_d([0,\infty)\times\tilde{W}^s_d)\subset W^s_d$, and a constant $\tilde{c}>0$ such that for all $\psi\in\tilde{W}^s_d$ and all $t\ge0$,
\begin{equation*}
|\Sigma_d(t,\psi)-\bar{\phi}_d|_{d,1}\le\tilde{c}\,e^{-\gamma t}|\psi-\bar{\phi}_d|_{d,1}.
\end{equation*}

(II) There exists a constant $\bar{c}>0$ such that each $\psi\in X_d$ with $[0,\infty)\times\{\psi\}\subset\Omega_d$ and
\begin{equation*}
e^{\beta t}|\Sigma_d(t,\psi)-\bar{\phi}_d|_{d,1}<\bar{c}\quad\text{for all}\quad t\ge0
\end{equation*}
belongs to $W^s_d$.

\medskip

The codimension of $W^s_d$ in $C^1_d$ is equal to 
\begin{equation*}
n+\dim\,Y_{d,c}+\dim\,Y_{d,u}=n+\dim\,C_{d,c}+\dim\,C_{d,u}.
\end{equation*}
As the continuous linear map $R_{d,1}:C^1\to C^1_d$ is surjective we can apply Proposition 7.1  and  obtain an open neighbourhood
$V$ of $ \bar{\phi}$ in $N\subset U\subset C^1$ so that
\begin{equation*}
W^s=W^s(\bar{\phi})=V\cap R_{d,1}^{-1}(W^s_d)
\end{equation*}
is a $C^1_F$-submanifold of $C^1$ with codimension $n+\dim\,C_{d,c}+\dim\,C_{d,u}$ and tangent space
\begin{equation*}
T_{\bar{\phi}}W^s=R^{-1}_{d,1}(T_{\bar{\phi}_d}W^s_d)=R^{-1}_{d,1}(Y_{d,s})=Y_s.
\end{equation*}
The next proposition shows that $W^s$ is the desired local stable manifold of $\Sigma$ at $\bar{\phi}$.

\begin{prop}
(i) $W^s\subset X$, and $W^s$ is locally positively invariant.

\medskip

(ii) There are an open neighbourhood $\tilde{V}$ of $\bar{\phi}$ in $V$ with $[0,\infty)\times(\tilde{V}\cap W^s)\subset\Omega$ and a constant $\tilde{c}>0$
such that for all $\phi\in\tilde{V}\cap W^s$ the solution $x:\R\to\R^n$ on $[0,\infty)$ of Eq. (1.1) with $x_0=\phi$ satisfies
\begin{equation*}
|x(t)-\bar{\phi}(0)|+|x'(t)|\le\tilde{c} e^{-\gamma t}|R_{d,1}\phi-\bar{\phi}_d|_{d,1}\quad\text{for all}\quad t\ge0.
\end{equation*}
(iii) There are an open neighbourhood $\hat{V}$ of $\bar{\phi}$ in $V$ and a constant $\hat{c}>0$ such that for every solution $x:\R\to\R^n$ on $[0,\infty)$ of Eq. (1.1) with $x_0\in\hat{V}\cap X$ and
\begin{equation*}
|x(t)-\bar{\phi}(0)|+|x'(t)|\le\hat{c}\,e^{-\beta t}\quad\text{for all}\quad t\ge0
\end{equation*}
we have $x_0\in W^s$.
\end{prop}

Proposition  15.1  is proved exactly as  \cite[Propositions 4.1, 4.2]{W8}, using the properties of $W^s_d$ stated above.

\section{The local unstable manifold}

In this section all segments  $x_t$ are defined on $(-\infty,0]$.
Fix some $a>0$ and consider the Banach spaces $B_a\subset C$ and  $B^1_a\subset C^1$ introduced in Section 1.  It is easy to see that the linear inclusion maps
\begin{equation*}
j_0:B_a\to C\quad\text{and}\quad j_1:B^1_a\to C^1
\end{equation*}
are continuous, as well as the restriction and prolongation maps 
\begin{equation*}
R_{a,d,1}:B^1_a\ni\phi\mapsto R_{d,1}\phi\in C^1_d\quad\text{and}\quad P_{a,d,1}:C^1_d\ni\chi\mapsto P_{d,1}\chi\in B^1_a.
\end{equation*}
The set $U_a=j_1^{-1}(N)\cap R_{a,d,1}^{-1}(U_d)\subset B^1_a$ is open and contains $\bar{\phi}$, and the $C^1_F$-map
\begin{equation*}
f_a:U_a\to\R^n,\quad f_a(\phi)=f(j_1\phi),
\end{equation*}
satisfies $f_a(\bar{\phi})=0$. Notice that every solution of the equation
\begin{equation}
x'(t)=f_a(x_t)
\end{equation}
on some interval also is a solution of Eq. (1.1) on this interval. The proof of \cite[Proposition 5.1]{W8} remains valid. Therefore we have
\begin{equation}
f_a(\phi)=f_a(\psi)\,\,\text{for all}\,\,\phi\in U_a,\,\,\psi\in U_a\,\,\text{with}\,\,\phi(s)=\psi(s)\,\,\text{on}\,\,[-d,0],
\end{equation}
each derivative $Df_a(\phi):B^1_a\to\R^n$, $\phi\in U_a$, has a linear extension $D_ef_a(\phi):B_a\to\R^n$, and the map
$$
U_a\times B_a\ni(\phi,\chi)\mapsto D_ef_a(\phi)\chi\in\R^n
$$
is continuous. Now results from \cite{W3} show that $X_a=\{\phi\in U_a:\phi'(0)=f_a(\phi)\}$ is a $C^1_F$-submanifold
of $B^1_a$, that the solutions of Eq. (16.1)
define a continuous semiflow $\Sigma_a:\Omega_a\to X_a$ on $X_a$, and that there is a local unstable manifold
$W^u_a\subset X_a$ at the stationary point $\bar{\phi}\in W^u_a$. $W^u_a$ is a $C^1_F$-submanifold of $B^1_a$ consisting of data $\phi\in X_a$ which are solutions of Eq. (16.1) on $(-\infty,0]$ with $\phi_s\to\bar{\phi}$ as $s\to-\infty$, and
$$
T_{\bar{\phi}}W^u_a=Y_u.
$$
(In order to verify the last equation observe that in \cite{W3} the tangent space of $W^u_a$ at $\bar{\phi}$ is obtained as the vector space of all maps
$\hat{\chi}:(-\infty,0]\to\R^n$ with $\hat{\chi}_0=\chi\in C_{d,u}$ which for some $t>0$ and for all integers
$j<0$ satisfy
$$
\hat{\chi}_{jt}=\Lambda^{-j}\chi
$$
where $\Lambda:C_{d,u}\to C_{d,u}$ is the isomorphism whose inverse is given by $T_{d,e,t}$. The maps in the vector space $Y_u=IC_{d,u}$ share the said property. The dimension of both vector spaces equals $\dim\,C_{d,u}$.)

\medskip

Moreover, there exist $\overline{\beta}>\overline{\gamma}>0$ and $c_u>0$ so that

\medskip

(I) $|\phi_s-\bar{\phi}|_{a,1}\le c_ue^{\bar{\beta}s}|\phi-\bar{\phi}|_{a,1}$ for all $\phi\in W^u_a$ and $s\le0$,

\medskip

and

\medskip

(II) for every solution $\psi\in B^1_a$ of Eq. (16.1) on $(-\infty,0]$ with
$$
\sup_{s\le0}|\psi_s-\bar{\phi}|_{a,1}e^{-\bar{\gamma}s}<\infty
$$
there exists $s_{\psi}\le0$ with $\psi_s\in W^u_a$ for all $s\le s_{\psi}$.

\medskip

From a manifold chart at $\bar{\phi}$ we obtain $\epsilon>0$ and a $C^1_F$-map
$$
w^u_a:Y_u(\epsilon)\to B^1_a,\quad Y_u(\epsilon)=\{\phi\in Y_u:|\phi|_{a,1}<\epsilon\},
$$
with $w^u_a(0)=\bar{\phi}$, $w^u_a(Y_u(\epsilon))$ an open subset of $W^u_a$, and $Dw^u_a(0)\eta=\eta$ for all $\eta\in Y_u$. Proposition 7.2 applies to the $C^1_F$-map $j_1\circ w^u_a$. So we may assume that
$$
W^u=W^u(\bar{\phi})=j_1w^u_a(Y_u(\epsilon))
$$
is a $C^1_F$-submanifold of the Fr\'echet space $C^1$ with 
$$
T_{\bar{\phi}}W^u=j_1Dw^u_a(0)Y_u=Y_u.
$$
The proof of \cite[Proposition 5.2]{W8} remains valid in the present setting. We state the result about the properties of the local unstable manifold $W^u$ as follows.

\begin{prop}
(i) Every $\phi\in W^u$ is a solution of Eq. (1.1) on $(-\infty,0]$, with  $\phi_s\to\bar{\phi}$ as $s\to-\infty$, and for all $s\le0$, 
$$
|\phi(s)-\bar{\phi}(0)|\le  c_ue^{\bar{\beta}s}|\phi-\bar{\phi}|_{a,1}\quad\text{and}\quad|\phi'(s)|\le 
c_ue^{\bar{\beta}s}|\phi-\bar{\phi}|_{a,1}.
$$
(ii) For every $\psi\in X$ which is a solution of Eq. (1.1) on $(-\infty,0]$ with
$$
\sup_{s\le0}e^{-\bar{\gamma}s}|\psi(s)-\bar{\phi}(0)|<\infty\quad\text{and}\quad\sup_{s\le0}e^{-\bar{\gamma}s}|\psi'(s)|<\infty
$$
there exists $s(\psi)\le0$ with $\psi_s\in W^u$ for all $s\le s(\psi)$.
\end{prop}

\section{Local center manifolds}

In this section we assume
$$
\{0\}\neq Y_c
$$
which is equivalent to
$$
\{0\}\neq C_{d,c}.
$$
In the sequel we recall the steps which in \cite[Section 6]{W8} led to a local center manifold at $\bar{\phi}$ which is $C^1_{MB}$-smooth, and point out the observation which yields $C^1_F$-smoothness.  

\medskip

The approach from \cite[Section 6]{W8} first follows constructions from the proof of \cite[Theorem 2.1]{K2} which were done for the case $\bar{\phi}_d=0$. Therefore we introduce
$V_d=U_d-\bar{\phi}_d$ and the $C^1_F$-map $g_d:V_d\to\R^n$. Then $g_d(0)=0$ and $Dg_d(0)=Df_d(\bar{\phi}_d)$. 

\medskip

There is a decomposition
$$
C^1_d=C^1_{d,s}\oplus C_{d,c}\oplus C_{d,u},\quad C^1_{d,s}=C^1_d\cap C_{d,s},
$$
into closed subspaces which defines a projection $P^1_{d,c}:C^1_d\to C^1_d$ onto $C_{d,c}$, and there is a norm $\|\cdot\|_{d,1}$ on $C^1_d$ which is equivalent to $|\cdot|_{d,1}$ and whose restriction to $C_{d,c}\setminus\{0\}$ is
$C^{\infty}$-smooth. 

\medskip

Next there exists $\Delta>0$ with
$$
N_{\Delta}=\{\phi\in C^1_d:\|\phi\|_{d,1}<\Delta\}
$$
contained in $V_d$ so that the restricted remainder map
$$
N_{\Delta}\ni\phi\mapsto g_d(\phi)-Dg_d(0)\phi\in\R^n
$$ 
has a global continuation
$$
r_{d,\Delta}:C^1_d\to\R^n
$$
with Lipschitz constant
$$
\lambda=\sup_{\phi\neq\psi}\frac{\|r_{d,\Delta}(\phi)-r_{d,\Delta}(\psi)\|_{d,1}}{\|\phi-\psi\|_{d,1}}\quad <\quad1.
$$

\medskip

The desired local center manifold at $\bar{\phi}\in X$ will be given, up to translation, by segments $(-\infty,0]\to\R^n$ of solutions on $\R$ of the equation
\begin{equation}
x'(t)=Dg_d(0)x_t+r_{d,\Delta}(x_t)\quad\text{(with segments in}\quad C^1_d)
\end{equation}
which do not grow too much at $\pm\infty$.

\medskip

For $\eta>0$ let $C^1_{d,\eta}$ denote the Banach space of all continuous maps $u:\R\to C^1_d$ with
$$
\sup_{t\in\R}e^{-\eta|t|}|u(t)|_{d,1}<\infty
$$
and the norm given by the preceding supremum. There exists $\eta_1>0$ so that for every $\phi\in C_{d,c}$ there is a unique continuously differentiable map
$$
x^{[\phi]}:\R\to\R^n
$$
which satisfies Eq. (17.1) for all $t\in\R$ and $P^1_{d,c}x^{[\phi]}_0=\phi$ and has the continuous map $\R\ni t\mapsto x^{[\phi]}_t\in C^1_d$  contained in the space $C^1_{d,\eta_1}$.
Observe that we have
$$
x^{[0]}(t)=0\quad\text{for all}\quad t\in\R.
$$

\medskip

Incidentally, from here on the proof in \cite[Section 6]{W8} deviates from the approach in \cite{K2}.

\medskip

Now consider the map
$$
J:C_{d,c}\ni\phi\mapsto\bar{\phi}+x^{[\phi]}|_{(-\infty,0]}\in C^1.
$$
Observe that the proof of \cite[Corollary 6.2 ]{W8} shows that the map $J$ is in fact  $C^1_F$-smooth,  not only $C^1_{MB}$-smooth, and
$$
DJ(0)\phi=I\phi\quad\text{for all}\quad\phi\in C_{d,c}.
$$
As $C_{d,c}$ is finite-dimensional and as $I$ is injective Proposition 7.2 yields an open neighbourhood $N_{d,c}$ of $0$ in $C_{d,c}$ so that the image
$$
W^c=J(N_{d,c})
$$
is a $C^1_F$-submanifold of the Fr\'echet space $C^1$, with
$$
T_{\bar{\phi}}W^c=IC_{d,c}=Y_c.
$$
By continuity of $J$ and $J(0)=\bar{\phi}$ we may assume $J(N_{d,c})\subset N\subset U$.
By continuity of the map
$$
C_{d,c}\ni\phi\mapsto R_{d,1}(J(\phi)-\bar{\phi})\in C^1_d
$$
at $0\in C_{d,c}$ we also may assume that for all $\phi\in N_{d,c}$ we have 
$$
\|x^{[\phi]}_0\|_{d,1}<\Delta\quad\text{for all}\quad \phi\in N_{d,c}
$$ 
or, $x^{[\phi]}_0\in N_{\Delta}$ for all $\phi\in N_{d,c}$, with segments 
$x^{[\phi]}_0$ defined on $[-d,0]$.

\medskip

We take $W^c$, with tangent space $Y_c$ at the stationary point $\bar{\phi}\in X$, as the desired local center manifold of the semiflow $\Sigma$ and verify that it has the appropriate properties. 
Following the proof of \cite[Proposition 6.3]{W8} we get
$$
W^c\subset X.
$$
Next, choose an open neighbourhood $U_{\ast}$ of $\bar{\phi}$ in $N\subset U$ so small that 
$$
R_{d,1}U_{\ast}\subset U_d\cap(N_{\Delta}+\bar{\phi}_d)
$$
and for all $\psi\in U_{\ast}$,
$$
P^1_{d,c}R_{d,1}(\psi-\bar{\phi})\in N_{d,c}.
$$
Then the proofs of \cite[Proposition 6.4, Proposition 6.5]{W8} remain valid. We state the result as follows.

\begin{prop}
(i) (Local positive invariance) For every  $(t,\psi)\in\Omega$ with $\psi\in W^c\subset X$ and
$\Sigma([0,t]\times\{\psi\})\subset U_{\ast}$ we have $\Sigma([0,t]\times\{\psi\})\subset W^c$.

\medskip

(ii) For every solution $y:\R\to\R^n$ of Eq. (1.1) on $\R$ with $y_t\in U_{\ast}$ for all $t\in\R$ we have $y_t\in W^c$ for all $t\in\R$.
\end{prop}

Observe that the proofs of both parts of Proposition 17.1   make use of \cite[Lemma 7.1]{W8} on uniqueness for an initial value problem with data in $C^1_d$.


\end{document}